\documentclass[11pt]{article}

\usepackage{authblk}

\usepackage{hyperref}
\usepackage{amsfonts}
\usepackage{graphicx}
\usepackage{tabularx}
\usepackage{array}
\usepackage[usenames,dvipsnames]{color}
\usepackage{comment}
\usepackage{amsmath}
\usepackage{amsthm}
\usepackage{amssymb}
\usepackage{fullpage}
\usepackage[dvipsnames]{xcolor}
\usepackage{tikz}
\usepackage{listings}
\usepackage{dsfont}
\usepackage{enumerate}
\usepackage{romannum}

\newtheorem{theorem}{Theorem}[section]

\newtheorem{conjecture}[theorem]{Conjecture}

\newtheorem{question}[theorem]{Question}

\newtheorem{lemma}[theorem]{Lemma}
\newtheorem{corollary}[theorem]{Corollary}
\theoremstyle{definition}

\def\epsilon{\varepsilon}
\DeclareMathOperator{\dist}{dist}

\title{Cops and Robbers on Graphs without Long or Short Holes}
\author[1]{Alexander Clow\,\thanks{Supported by the Natural Sciences and Engineering Research Council of Canada (NSERC) through PGS D-601066-2025}}
\date{}

\begin{document}
\pagenumbering{arabic}

\maketitle

\begin{abstract}
    In 2019, Sivaraman conjectured that every $P_k$-free graph has cop number at most $k-3$ for $k\geq 5$. Chudnovsky, Norin, Seymour, and Turcotte proved $P_5$-free graphs have cop number at most $2$. For $k\geq 6$ Sivaraman's conjecture remains widely open. There is little evidence to believe Sivaraman's conjecture is optimal, making the discovery of substantially better bounds interesting, even for subclasses of $P_k$-free graphs. We consider this problem for graphs with no long holes, which generalizes $P_k$-free graphs, that also have some forbidden shorter holes. Specifically, we show that for $t\geq 13$, $(C_4,\dots, C_{t}, C_k,C_{k+1},\dots)$-free graphs have cop number at most $\lceil \frac{k-2}{\lfloor \frac{t}{4}\rfloor }\rceil+1$. Taking $C_{k-1}$, which has cop number $2$ whenever $k\geq 5$, demonstrates that this bound is tight up to a small constant factor for infinity many pairs $k$ and $t$. The proof relies on showing that in graphs without short holes, the closed neighbourhood of a short induced path can be guarded by $1$ cop, even when these paths are not geodesics. Additionally, for $t\geq 2$ we show that $(C_{t+2},C_{t+3},C_k, C_{k+1},\dots, )$-free graphs have cop number at most $(t-1)\lceil \frac{(k-2)}{t}\rceil$. This show that $(C_4,C_5,C_k,C_{k+1},\dots)$-free graphs have cop number at most $\lceil \frac{k-2}{2}\rceil$, extending work by Clow and Meger on $(P_k,\text{claw}, \text{butterfly}, C_4,C_5)$-free graphs to a substantially more general class.
\end{abstract}

\section{Introduction}

Cops and Robbers is a two-player game played on a connected graph, see \cite{AIGNER1984, nowakowski1983vertex,quilliot1983problemes}. 
To begin the game, the cop player places $k$ cops onto vertices of the graph,
then the robber player chooses a vertex to place the robber. 
Players take turns moving. 
During the cop player’s turn, each cop either moves to an adjacent vertex or passes and remains at their current vertex. 
Similarly, on the robber player’s turn, the robber either moves to an adjacent vertex or passes and remains at their current vertex.
The cop player wins if after finitely many moves a cop can move onto the vertex occupied by the robber, called capturing.
The robber player wins if the robber can provide a strategy to evade capture indefinitely. 
The least number of cops required for the cop player to win, regardless of the robber’s strategy, is called  the cop number of a graph, denoted $c(G)$ for a graph $G$.
If $c(G)\leq k$, then we say $G$ is $k$-cop win.
We suppose all graphs are connected, unless stated otherwise.
For more background on Cops and Robbers we recommend \cite{bonato2011game}.

We define the graphs $P_t$, $C_t$, and $K_t$ as the path with $t$ vertices, the cycle with $t$ vertices, and the complete graph on $t$ vertices, respectively. 
We call the complete bipartite graph $K_{1,3}$ the \text{claw}.
The $E$ graph is given by subdividing two edges of the $\text{claw}$.
If $G$ and $H$ are graphs, then we denote the disjoint union of $G$ and $H$ by $G+H$.
We let $\alpha(G)$ denote the independence number of $G$.
For more background and definitions in graph theory we refer the reader to \cite{west2001introduction}.

It is standard to consider classes of graphs defined by forbidden substructures such as minors or induced subgraphs.
A graph $G$ is $H$-free or $H$-minor free if $G$ does not
contain, respectively, any induced subgraph or minor which is isomorphic to $H$. 
Of particular importance for this paper will be the class of $P_k$-free graphs.
If $\mathcal{H}$ is a (finite or infinite list) list of graphs, we say $G$ is $\mathcal{H}$-free if $G$ has no graph in $\mathcal{H}$ as an induced subgraph.
At various points in the paper we will take $\{C_{k}, C_{k+1},\dots \}$ to mean the set of all cycles with length at least $k$.
Notice that if $G$ is $P_{k-1}$-free, then $G$ is also $(C_k,C_{k+1},\dots )$-free.
For $k\geq 4$,
a hole of length $k$ in $G$ is an induced copy of $C_k$ in $G$.

Cops and Robbers has been studied extensively in relation to forbidden minors and forbidden induced subgraphs.
Andreae \cite{andreae1986pursuit} showed that for all graphs $H$, there exists a constant $m = m(H)$
such that if $G$ has no $H$-minor, then $c(G) \leq m$.
The constant $m = m(H)$ here was recently improved by Kenter, Meger, and Turcotte \cite{kenter2025improved},
particularly for small or sparse graphs $H$.
When forbidding an induced subgraph $H$, or a subgraph $H$, one cannot guarantee the cop number of $H$-free graphs is bounded.
In fact, it was shown by Joret, Kami{\'n}ski, and Theis \cite{joret2010cops} 
that $H$-free graphs have bounded cop number if and only if $H$ is a linear forest (i.e a disjoint union of paths).

In \cite{joret2010cops} Joret, Kami{\'n}ski, and Theis demonstrated that 
every connected $P_k$-free graph is $(k-2)$-cop win.
This led Sivaraman \cite{sivaraman2019application} to give a different and simpler proof that every connected $P_k$-free graph is $(k-2)$-cop win.
This simpler proof involves the well known Gy\'arf\'as path argument,
and has the added benefit of bounding the number of turns it takes the cops to catch the robber.
Sivaraman generalized this in \cite{sivaraman2019cop}
to prove that  $(C_k,C_{k+1},\dots )$-free graphs have cop number at most $k-3$.
This improved an earlier theorem by Kosowski, Li, Nisse, and Suchan \cite{kosowski2015k} who showed $k-1$ cops 
are sufficient.

When $k$ is large there is no evidence that $k-2$ is the 
optimal upper bound for the cop number of $P_k$-free graphs.
This lead Sivaraman to the following conjecture.
Notice that $k\geq 5$ is required here because $C_4$
is a $P_4$-free graph with cop number $2$.

\begin{conjecture}[Sivaraman \cite{sivaraman2019application}]\label{Conj: main}
    For all $k\geq 5$, if $G$ is $P_k$-free, then $c(G) \leq k - 3$.
\end{conjecture}

This conjecture has seen a significant amount of attention,
most of which has been focused on the $P_5$-free case.
Initially many authors, see \cite{gupta2023cops,liu2019cop,sivaraman2019copK2,turcotte2022cops}, solved various subcases of the $P_5$-free case.
The general $P_5$-free conjecture was proven
by
Chudnovsky, Norin, Seymour, and Turcotte in \cite{chudnovsky2024cops}.
For several subclasses  $P_k$-free classes,
such as $(P_k,\text{claw})$-free graphs and $(P_k,C_t)$-free graphs where $t\leq k$,
Liu \cite{liu2019cop} showed
$k-3$ cops are sufficient.
Some of these results were significantly improved
by Clow and Meger in \cite{clow2027cops}.
For example,
\cite{clow2027cops}
showed that $(P_k,E)$-free graphs have cop number at most $\Big\lceil \frac{k-1}{2}\Big\rceil+3$.
During the preparation of this manuscript 
the additive constant $3$ was removed from this bound by
Gupta, Pradhan, and Swain \cite{PkNoClawFork2026} in the special case of $(P_k,\text{claw})$-free graphs.
Notice that the class of $(P_k,\text{claw})$-free graphs is a strict subset of the class of 
$(P_k,E)$-free graphs.

Much less attention has been payed to 
demonstrating the existence of $P_k$-free graphs with large cop number.
This is part of a general trend in the literature
where there are not many known techniques to lower bound cop number.
The techniques that are known mostly apply to graphs with large girth,
see for instance \cite{AIGNER1984,bradshaw2023cop,frankl1987cops},
or in some cases graphs forbidding small subgraphs such as $C_4$, see \cite{bonato2013cops}.
Neither of these restrictions is very amenable
to constructing interesting $P_k$-free graphs.

The best and, to the author's knowledge, only example 
of a $P_k$-free graph in the literature with large cop number comes from Clow and Zaguia \cite{clow2026cops}.
They demonstrate, for all $k$, the existence of graphs with cop number and clique cover number $k$.
Any graph with clique cover number $k$
is trivially $P_{2k+1}$-free.
These graphs are constructed randomly,
and this work was motivated by a problem 
proposed by Turcotte in \cite{turcotte2022cops}
which asked if for all $k$ there exists graphs $G$ with $k=c(G) = \alpha(G)$.
Note that the cop number of random graphs have been studied in other situations, 
see for instance \cite{luczak2010chasing,pralat2010does,pralat2016meyniel}.
Because of the random nature of the construction in \cite{clow2026cops}
these results do not provide lower bounds for $(P_k,H)$-free graphs.

We contribute to this literature by studying the
cop number of graphs with no long holes that have a further 
restriction on which short holes they admit.
Recalling that graphs with no long hole generalize $P_k$-free graphs
this second restriction is required,
otherwise this work would have direct implications to proving Conjecture~\ref{Conj: main}.
We also note that this direction of study is motivated for reasons outside of Conjecture~\ref{Conj: main}.
For instance, graphs with no holes (i.e. chordal graphs) are well known to be cop win.
Meanwhile, graphs with no even holes and no $\text{claw}$
have been shown to have cop number at most $2$ 
\cite{javadi2021game}.
Moreover, Clow and Zaguia \cite{clow2026cops}
proved that any graphs with cop number and clique cover number $k$
have a hole of length $t$ for all $4\leq t\leq k+1$.
The latter of which has implications to the cop number of perfect graphs.

We start by providing 
a bound on the cop number of 
graphs with no long hole and no short hole.
This result is primarily interesting because, using Liu's work in \cite{liu2019cop},
one can only show that forbidding $t$ holes as induced subgraphs will allow $k-3-t$ cops to win.
Whereas we show that by forbidding the shortest $t$ holes $O(\frac{k}{t})$ cops suffice.

Beyond simply demonstrating this bound,
the proof may be of independent interest.
Perhaps even more interest than the bound itself.
Like many Cops and Robbers proof, the argument proceeds by showing that the cops can 
monotonically expand their territory by guarding subgraphs.
However, unlike the standard version of this argument,
where cops guard shortest paths,
we use cops to guard the closed neighbourhood of
arbitrary short induced paths.
To the author's knowledge this is the first argument of its kind.

\begin{theorem}\label{Thm: No short hole}
    Let $t\geq 13$ be an integer.
    Every $(C_4,\dots, C_{t}, C_k, C_{k+1,\dots})$-free
    graph $G$ has cop number
    \[
    c(G)  \leq   { \Big\lceil \frac{k-2}{\lfloor\frac{t}{4} \rfloor} \Big\rceil}+1.
    \]
\end{theorem}

Observe that this bound is tight within a small factor for infinitely some pairs $k$ and $t$.
For example, take any for any $k\geq 15$ with $k\equiv 2 \mod{4}$. 
The cycle
$C_{k-1}$ is $(C_4,\dots, C_{k-2},C_k,C_{k+1},\dots)$-free
and $C_{k-1}$ has cop number $2$, whereas Theorem~\ref{Thm: No short hole}
implies an upper bound of $5$.
It would be interesting to understand for what other 
values of $t$ as a function of $k$
the bound in Theorem~\ref{Thm: No short hole} is within a constant factor of being tight.

Our next theorem 
is aimed at showing that even if a smaller list of short holes are forbidden,
one can still obtain a multiplicative improvement over Conjecture~\ref{Conj: main}.
Also, recall that Clow and Meger \cite{clow2027cops} proved that 
$(P_k,\text{claw}, \text{butterfly}, C_4,C_5)$-free graphs have cop number
at most $\lceil \frac{k-1}{3}\rceil +3$.
Hence, this theorem shows 
that one can still multiplicatively improve Conjecture~\ref{Conj: main}
if the \text{claw} and \text{butterfly} are removed from this list.
Of course the removal of these graphs leads to a worst bound.

\begin{theorem}\label{Thm: general n56}
    Let $t\geq 2$ be an integer.
    Every $(C_{t+2}, C_{t+3}, C_k, C_{k+1,\dots})$-free graph $G$ has
    cop number
    \[
    c(G) \leq  (t-1)\Big\lceil \frac{k-2}{t} \Big\rceil.
    \]
\end{theorem}

\begin{corollary}
    If $G$ is $(P_k,C_4,C_5)$-free
    graph, then $c(G) \leq \lceil \frac{k-1}{2}\rceil$.
\end{corollary}

The rest of the paper is organized as follows.
Section~2 is used to prove Theorem~\ref{Thm: No short hole}
while Section~3 proves Theorem~\ref{Thm: general n56}.
We conclude with a discussion of future work where some conjectures are made.

\section{Graphs with no Short Holes}
\label{sec: No short hole}

The goal of this section is to prove Theorem~\ref{Thm: No short hole}.
Our argument adopts a common motif in Cops and Robbers by having cops
guard certain subgraphs as part of the cop's strategy to monotonically reduce the robber's territory.
However, unlike most arguments of this kind, we will not have cops guard geodesic paths.
We provide a more nuanced discussion of how these elements 
of the proof are similar and distinct from standard methods later in this section.

We start by noting some facts regarding how vertices in the open neighbourhood of an induced path
can be connected by short paths in 
the rest of the graph.
We will leverage these facts later, specifically in Lemma~\ref{Lemma: 1-guard},
to show that $1$-guard can protect the closed neighbourhood of a short induced path
even if that path is not geodesic.

In order to state our results we must introduce some non-standard notation regarding distances.
Given a graph $G$ and an induced subgraph $H$ of $G$, let $\dist_H(x,y)$ be the length of the shortest $(x,y)$-path in
$G$ all of whose internal vertices and edges are in $H$.
Thus, if $x,y$ are both in $V(G)\setminus V(H)$, then $\dist_H(x,y)> 1$.
Notice this does not require $x$ or $y$ to be vertices in $H$,
whereas one would normally require this.

\begin{lemma}\label{Lemma: Common neighbourhood}
    Let $t\geq 4$ be an integer,
    $G$ a $(C_4,\dots, C_{t})$-free
    graph, let $P$ be an induced path in $G$ of order $\ell < t-1$,
    and let $H$ be $G-N[P]$.
    If $x,y \in N(P)$ are non-adjacent, then
    $\dist_{H}(x,y)\geq t-\ell$.
    Moreover, if $x,y \in N(P)$ are adjacent, 
    then $\dist_{H}(x,y)\geq  t$ or $\dist_H(x,y) = 2$.
\end{lemma}

\begin{proof}
    Let all objects be defined as in the lemma statement
    with $P = p_1\dots p_\ell$.
    Suppose for contradiction that 
    there exists non-adjacent vertices $x,y \in N(P)$
    such that $\dist_{H}(x,y)< t-\ell$.
    Since $\dist_{H}(x,y)< t-\ell$ and $x,y$ are non-adjacent there exists an induced $(x,y)$-path $Q$ in $H$ of order
    less than $t-\ell$ and order at least $3$.
    Let $Q = x q_1\dots q_c y$ be such a path.

    Note that $Q$ is an induced path, since it is a geodesic in an induced subgraph of $G$.
    If
    $N(x)\cap N(y)\cap V(P) \neq \emptyset$,
    then let $p \in N(x)\cap N(y)\cap V(P)$.
    In this case $pQ$ is a cycle of length between $4$ and $t-\ell + 1 < t$.
    But no such hole exists by assumption.
    Suppose then without loss of generality that
    $p_i$ is the largest indexed neighbour of $x$ in $P$, and $p_j$ is the smallest indexed neighbour of
    $y$ in $P$,
    and that $i<j$.
    Immediately, $C = x q_1\dots q_c y p_j p_{j-1}\dots p_i$ is an induced
    cycle with length $4\leq (j-i+1)+|V(Q)|\leq (t-\ell) + |V(P)|= t$.
    But this is a contradiction since $G$ is $(C_4,\dots, C_t)$-free.
    we conclude that if $x$ and $y$ are non-adjacent, then
    $\dist_{H}(x,y)\geq t-\ell$.

    Now, suppose that $x,y$ are adjacent 
    and $2 < \dist_{H}(x,y)\leq t-1$.
    Then,
    there exists an induced $(x,y)$-path $Q$ in $H$ of order
    at most $t$ and order at least $4$.
    Since $x,y$ are adjacent the vertices of $Q$ induce a cycle in
    $G$ of length $3 < |Q|+1 \leq t$, but this contradicts that $G$ is $(C_4,\dots, C_t)$-free.
    This completes the proof.
\end{proof}

\begin{lemma}\label{Lemma: Helly neighbourhood}
    Let $t\geq 4$ be an integer,
    let $G$ a $(C_4,\dots, C_{t})$-free
    graph, let $P$ be an induced path in $G$ of order $\ell < t-1$,
    let $H$ be $G-N[P]$,
    and let $v \in V(G)\setminus N[P]$.
    If $W$ is the set of all vertices $w \in N(P)$ with $\dist_{H}(w,v) < \frac{t-\ell}{2}$,
    then $W$ is a clique and there exists a vertex $u \in V(P)$ such that $W\subseteq N(u)$.
\end{lemma}

\begin{proof}
    Let all objects be defined as in the lemma statement.
    Let $P = v_1,\dots, v_\ell$.
    Suppose there are two distinct vertices $w,z \in W$. 
    By the definition of $W$, for any distinct vertices $w,z \in W$ 
    it is true that $\dist_{H}(w,z) < t-\ell$.
    Hence, by Lemma~\ref{Lemma: Common neighbourhood} we have that $(w,z)\in E(G)$.
    Since $w,z\in W$ were selected without loss of generality $W$ is a clique, as required.

    Since $P$ is an induced path of length less than $t-1$ and $G$ is 
    $(C_4,\dots, C_{t})$-free, for all $w \in N(P)$, it is easy to see that
    $N(w)\cap V(P)$ induces connected subpath of $P$.
    Thus, the set $X = \{N(w)\cap V(P): w \in W\}$ satisfies the Helly property, 
    since any set of subintervals of a line satisfy the Helly property.
    Recall that the Helly property for a set of set $S$ states that if all elements of a subset $X \subseteq S$
    pairwise intersect, then there exists $x \in \bigcap_{A\in X} A$.
    
    Now,
    let $x,y\in W$ be distinct vertices.
    We claim $N(x)\cap N(y)\cap V(P)\neq \emptyset$.
    Suppose otherwise for contradiction.
    Then, without loss of generality let $v_i$ be the largest index neighbour of $x$ and $v_j$
    be the smallest index neighbour of $y$, both in $P$,
    such that $i<j$.
    Then, $xp_ip_{i+1}\dots p_jy$ is a hole of length at least $4$
    and length at most $\ell+2\leq t$.
    This is a contradiction since $G$ is $(C_4,\dots, C_{t})$-free.
    Given $x,y$ were selected without loss of generality, we conclude that
    all sets in $X$ intersect.
    
    Since $X$ satisfies the Helly property, and for all pairs $A,B \in X$
    we have demonstrated $A\cap B \neq \emptyset$,
    we conclude that
    \[
    \bigcap_{x \in S} \Big(N(x)\cap V(P)\Big) \neq \emptyset.
    \]
    Therefore, there exists a vertex $u \in V(P)$
    such that $W \subseteq N[u]$ as desired.
    This concludes the proof.    
\end{proof}

As has already been mentioned, 
we will not guard paths in the standard manner originally introduced in \cite{AIGNER1984}.
The standard method involved $1$ cop protecting a shortest path, and has been widely applied, see for instance \cite{andreae1986pursuit,frankl1987cops,kenter2025improved}.
It is well known that a shortest path is always $1$-guardable regardless of the host graph. 
A subgraph $H$ in $G$ is \emph{guarded} by a set of cops, if the cops can from their current position,
move in such a way that if the robber ever enters $H$, then the cops will capture the robber on their next turn.

Path guarding arguments have been generalized in various ways.
For example \cite{de2024guarding} characterizes 
all geodesic subgraphs which are $1$-guardable, regardless of their host graph.
Notice that if one wants a subgraph to be guardable regardless of which graph it is hosted in,
then assuming the subgraph is geodesic is necessary.
Here a subgraph is \emph{geodesic} if the distances in the subgraph are preserved in the host graph.
Meanwhile other work, see for example
\cite{clow2025cops,de2024guarding, lehner20262},
explores how this kind of strategy can be extended to protect paths in settings where
the cop's movement is constrained.
Some quite recent work from \cite{bose2026cops}
has also provided machinery for efficiently guarding sets of vertices with small radius.

In Lemma~\ref{Lemma: 1-guard}
we take a different approach to generalizing these ideas.
Specifically, rather than constraining the cops or discussing subgraphs other than paths
we relax the assumption that our path is geodesic by restricting the structure of the host graph.
By doing so we will be able to guard paths, in fact the closed neighbourhood of paths, which are not necessarily geodesic.
This grants a massive advantage when working with dense graphs, 
noting of course that graphs with forbidden induced subgraphs tend to be denser
than many of the graphs normally considers in a path guarding argument.

\begin{lemma}\label{Lemma: 1-guard}
    Let $t\geq 12$ be an integer,
    let $G$ a $(C_4,\dots, C_{t})$-free
    graph,  let $P$ be an induced path in $G$ of order $\ell \leq \frac{t}{4}$, 
    and let $H$ be $G-N[P]$.
    If the robber currently occupies a vertex $r \in V(G)\setminus N[P]$,
    and there exists a cop $C$ in either the situation 
    \begin{enumerate}[(1)]
        \item it is the robber's turn to move and $C$ occupies a $u\in N(P)$ such that 
    $\dist_{H}(u,r)\leq \lfloor \frac{t}{2}-\ell-1\rfloor$, or 
    
        \item it is the cops turn to move and $C$ occupies a vertex in $P$ and $\dist_H(r,N(P)) \geq \lfloor \frac{t}{2}-\ell-1\rfloor$,
    \end{enumerate}
    then from this position $C$ can guard $N[P]$ indefinitely.
\end{lemma}

\begin{proof}
    Let all objects be defined as in the lemma statement with $P = p_1\dots p_\ell$.
    We suppose without loss of generality that the robber never moves adjacent to $C$.
    We claim that the cop $C$ can indefinitely prevent
    the robber from entering $N[P]$.
    For contradiction we suppose the robber has a strategy that allows them to enter $P$,
    as the cop $C$ can remain still until the robber moves,
    we suppose without loss of generality that the robber moves first from $r$ to a new vertex $w$.
    Trivially, the robber must enter $N(P)$ to enter $N[P]$ 
    so it is sufficient to show the cop $C$ can prevent the robber entering $N(P)$.
    We play until the robber first enters $N(P)$,
    meaning that the robber must play in $H$ for the duration of play we are considering.
    
    To begin, we observe that if $C$ is in state (1) or (2), then the robber cannot move into 
    $N(P)$ on their next move.
    This is obvious for state (2), since state (2) assumes 
    $\dist_H(r,N(P)) \geq \lfloor \frac{t}{2}-\ell-1\rfloor > 1$
    given $t\geq 12$.
    To observe that same in state (1), 
    let $R$ be the set of all vertices $v$ in $N(P)$
    with $\dist_{H}(r,v) \leq \lfloor \frac{t}{2}-\ell-1\rfloor$.
    By Lemma~\ref{Lemma: Helly neighbourhood} $R$ is a clique, and in state (1)
    $C$ occupies a vertex in $R$.
    Hence, the robber cannot move into $N(P)$ on their next turn if we are in state (1) or (2).
    Thus, it is sufficient to show that no matter how the robber plays, 
    the cop beginning in state (1) or (2)
    can return to state (1) or (2) without allowing the robber to enter $N(P)$ during the intervening turns.

    \vspace{0.25cm}
    \noindent\underline{Case.1:} The cop starts in state (1).
    \vspace{0.25cm}

    If the robber does not move, the cop can remain still and stay in state (1).
    Suppose then without loss of generality that the robber moves from $r$ to some vertex $w$.
    Let $R$ be defined as before, 
    and let $W$ be the set of all vertices $v$ in $N(P)$
    with $\dist_{H}(w,v) \leq \lfloor \frac{t}{2}-\ell-1\rfloor$.
    By Lemma~\ref{Lemma: Helly neighbourhood} we note that both $R$ and $W$ are cliques.
    If $R \cap W\neq \emptyset$, then $C$ moves from $u$ to a vertex in $R\cap W$, this may mean $C$ remains on $u$.
    After this move the cop has returned to state (1).
    Suppose then that $R \cap W = \emptyset$.

    Observe that if $\dist_{H}(r,N(P)) \leq \lfloor \frac{t}{2}-\ell-1\rfloor-1$,
    then it is impossible for $W\cap R=\emptyset$,
    by the definition of $R$ and $W$ and the fact that 
    $\dist_{H}(x,w) \leq \dist_{H}(x,r) + 1$.
    Then our assumption that $R$ is not empty, since $C$ began in state (1),
    and our assumption $R \cap W = \emptyset$, we can conclude that $\dist_{H}(r,N(P)) = \lfloor \frac{t}{2}-\ell-1\rfloor$.
    Next, we note that $\dist_{H}(w,N(P)) >  \lfloor \frac{t}{2}-\ell-1\rfloor$
    if and only if
    $W = \emptyset$.
    We consider the cases where $W$ is empty vs non-empty separately.

    \vspace{0.25cm}
    \noindent\underline{SubCase.1.1:} $W$ is non-empty.
    \vspace{0.25cm}

    Since $W$ is non-empty there exists $b\in W$.
    Let $b$ be such a vertex.
    Then, 
    $\dist_{H}(w,b) \leq \lfloor \frac{t}{2}-\ell-1\rfloor$
    and $\dist_H(u,w) = \dist_H(u,r)+1 \leq \frac{t}{2} - \ell$.

   If $(u,b) \notin E(G)$, then by the triangle inequality
   \begin{align*}
       \dist_H(u,b) & \leq \dist_H(u,w)+\dist_H(w,b) \\
       & \leq 1+ \dist_H(u,r)+\dist_H(w,b) \\
       & \leq 1+2\Big(\frac{t}{2} - \ell - 1\Big) \\
       & = t-2\ell -1 \\
       & < t-\ell.
   \end{align*}
    But Lemma~\ref{Lemma: Common neighbourhood} implies that $\dist_H(u,b)\geq t-\ell$
    since $(u,b) \notin E(G)$, a contradiction.
    
    Suppose then $(u,b) \in E(G)$.
    In this case the cop $C$ moves from $u$ to $b$.
    This returns $C$ to state (1) as required completing the subcase.

    \vspace{0.25cm}
    \noindent\underline{SubCase.1.2:} $W$ is empty.
    \vspace{0.25cm}

    In this case, the $C$ moves from $u$ to a vertex in $P$.
    This is possible since $u \in N(P)$.
    This moves $C$ into state (2) as required completing the subcase. 
    \hfill $\diamond$
    \vspace{0.25cm}
    
    \vspace{0.25cm}
    \noindent\underline{Case.2:} The cop starts in state (2).
    \vspace{0.25cm}

    Call the vertex of $P$ occupied by the cop $p_i$.
    From here $C$ does not move until the robber enters a vertex $z$
    such that $\dist_{H}(z,N(P)) \leq  \lfloor \frac{t}{2}-\ell-1\rfloor$.
    Suppose $z$ is the first such vertex that the robber has entered.
    Then, $\dist_{H}(z,N(P)) =  \lfloor \frac{t}{2} - \ell - 1\rfloor$
    and it is the cop's turn to move after the robber enters $z$.

    Let $Z$ be the set of all vertices $v$ in $N(P)$
    with $\dist_{H}(z,v) \leq \lfloor \frac{t}{2}-\ell-1\rfloor$.
    By our choice of $z$, the set $Z$ is not empty.
    By Lemma~\ref{Lemma: Helly neighbourhood} there is a vertex $p_j \in V(P)$
    such that $Z \subseteq N(p_j)$.
    If $i=j$, then the cop $C$ moves to a vertex of $Z$ returning the game to state (1).
    Suppose then without loss of generality that $i < j$.

    The cop $C$ moves to $p_j$ along a shortest path in $P$.
    Since the cop moves first this will take $j-i$ steps, 
    in which time the robber only gets to move $j-i-1$ times.
    By our assumption that $\ell \leq \frac{t}{4}$, it is true that
    \[
    j-i-1 \leq \ell-2 < \Big\lfloor \frac{t}{2} - \ell - 1\Big\rfloor.
    \]
    So the cop will arrive at $p_j$ before the robber can enter a vertex of $N(P)$.
    Notice that once $C$ arrives to $p_j$, the robber cannot enter $N(P)$ on their next turn
    because any vertex that they could reach in $j-i$ steps is in $Z$ and $Z \subseteq N(p_j)$.

    Let $m$ be the location of the robber when 
    the cop $C$ next get to move after arriving at $p_j$,
    and let $M$ be the set all vertices $v$ in $N(P)$
    with $\dist_{H}(m,v) \leq \lfloor \frac{t}{2}-\ell-1\rfloor$.
    If $M = \emptyset$ then the robber is at a distance greater than 
    $\lfloor \frac{t}{2}-\ell-1\rfloor$ from $N(P)$,
    meaning the game has returned to state (2).
    Otherwise $M$ is not empty. 
    
    If $Z \cap M \neq \emptyset$,
    then $C$ can move from $p_j$ to $M$,
    since $p_j$ was selected to satisfy $Z \subseteq N(p_j)$.
    If this happens, then 
    $C$ has returned to state (1).
    We suppose then that 
    $Z \cap M = \emptyset$.

    Suppose for contradiction there exists a pair of vertices $a \in Z$ and $b \in M$
    such that $(a,b) \notin E(G)$.
    It follows from Lemma~\ref{Lemma: Common neighbourhood}
    that $\dist_H(a,b)\geq t-\ell$.
    Recalling that $\dist_H(z,m) \leq j-i$ and $\dist_H(a,z) = \lfloor \frac{t}{2} - \ell - 1\rfloor$
    and $\dist_H(b,m) \leq \lfloor \frac{t}{2}-\ell-1\rfloor$
    we can see that the triangle inequality implies 
    \begin{align*}
        \dist_H(a,b) &\leq \dist_H(a,z)+\dist_H(z,m)+\dist_H(m,b)\\
        & \leq (j-i)+ 2\Bigg(\frac{t}{2} - \ell - 1\Bigg) \\ 
        & \leq (\ell-1) + t-2\ell-2 \\
        & = t-\ell-3 \\
        & < t-\ell.
    \end{align*}
    This is a contradiction.
    So we observe that $(a,b)\in E(G)$.
    Since $a$ and $b$ were selected without loss of generality 
    we conclude that $Z$ is complete to $M$.
    That is,
    every possible edge exists between $Z$ and $M$.

    It is the cop's turn, so $C$ moves from $p_j$ to any vertex $a$ in $Z$.
    Suppose the robber moves from $m$ to a vertex $q$.
    Since every possible edge exists from $Z$ to $M$, we note that $q\notin M$,
    so the robber has not moved into $N(P)$.

    Define $Q$ to be the set of all vertices $v$ in $N(P)$
    with $\dist_{H}(q,v) \leq \lfloor \frac{t}{2}-\ell-1\rfloor$.
    If $Q$ is empty, then $\dist_H(q,N(P))> \lfloor \frac{t}{2}- \ell -1\rfloor$, 
    the cop $C$ moves back to $p_j$ and we are in state (2) at the start of the cops' next turn.
    If $Q \cap M \neq \emptyset$, or $Q \cap Z \neq \emptyset$,
    then the cop at $a$ can move to some $x \in (Z\cup M)\cap Q$, which is possible since $M \cup Z\subseteq N[a]$. 
    This returns the cop to state (1) as required.
    Otherwise, $Q \cap (M \cup Z) = \emptyset$ and $Q \neq \emptyset$.
    Let $x \in Q$.

    If $(a,x) \in E(G)$, then $C$ moves from $a$ to $x$ and the game is in state (1).
    Otherwise, $(a,x)\notin E(G)$
    and as a result Lemma~\ref{Lemma: Common neighbourhood}
    implies that $\dist_H(a,x)\geq t-\ell$.
    But as before, 
    triangle inequality implies 
    \begin{align*}
        \dist_H(a,x) &\leq \dist_H(a,z)+\dist_H(z,m)+\dist_H(m,q)+\dist_H(q,x)\\
        & \leq 1+ (j-i)+ 2\Bigg(\frac{t}{2} - \ell - 1\Bigg) \\ 
        & \leq 1+ (\ell-1) + t-2\ell -2\\
        & = t-\ell-2 \\
        & < t-\ell.
    \end{align*}
    Since this is a contradiction, we observe that the cop $C$ can force the game back to state (1) or (2)
    without ever letting the robber enter $N(P)$ in the period between instances of state (1) or (2).
    \hfill $\diamond$
    \vspace{0.5cm}
    
    This completes the proof
\end{proof}

We are now prepared to prove Theorem~\ref{Thm: No short hole}.
we will use an adapted version of the Gy\'arf\'as path argument, similarly to that used in \cite{sivaraman2019cop}.
Our proof differs from \cite{sivaraman2019cop} in how we guard the induced paths involved,
which leads to us requiring substantially more subtlety in the argument.

\begin{proof}[Proof of Theorem~\ref{Thm: No short hole}]
Let $t\geq 13$ be an integer and let
    $G$ be a $(C_4,\dots, C_{t},C_k, C_{k+1},\dots)$-free
    graph. 
    We claim that  $\Big\lceil \frac{k-2}{\lfloor\frac{t}{4} \rfloor} \Big\rceil+1$
    cops have a winning strategy in $G$.
    Suppose otherwise for contradiction.
    The cop's strategy will proceed in two phases.
    To make the proof easier to read, we clearly delineate the phases.
    The game begins in Phase~1.
    
    \vspace{0.25cm}
    \noindent\underline{Phase~1:} 
    \vspace{0.25cm}

    All cops begin on the same vertex $s \in V(G)$.
    The cops' strategy in Phase~1 proceeds in stages with stage $1$ starting at the beginning of the game.
    Each stage starts with the cops' turn to move.
    Let $i\geq 1$ be fixed but arbitrary, we suppose that all stages $1$ to $i-1$ have occurred already,
    we suppose we are at the start of stage $i$.
    We suppose the game is in the following state:
    For all $1\leq j < i$, there is a distinct induced path $P^j$ of order $\lfloor\frac{t}{4} \rfloor$,
    such that $N[P^j]$ is guarded by $1$ cop who has no other responsibilities.
    All paths $P^j$ are vertex disjoint.
    Also, the union of all vertices $\cup_{j=1}^{i-1} V(P^j)$ induces a path, call it $Q$,
    while there is an endpoint $x$ of $Q$ which has a neighbour in the component of $G-N[Q\setminus \{x\}]$ containing the robber.
    Such a component must exist since we suppose the robber has not been captured and $G$ is connected.
    When $i=1$, all the previous statements hold vacuously if we take $s$ to be the only vertex in $Q$. 

    If $(i-1)\lfloor\frac{t}{4} \rfloor\geq k-2$, then we halt and proceed to Phase~2.
    Otherwise we proceed.
    Since we have not halted in step $i-1$,
    at least $1$ cop unassigned 
    to guard a path $P^j$ for $1 \leq j < i$.
    Let $C$ be such a cop.

    \vspace{0.25cm}
    \noindent\underline{Claim.1:} The cop $C$ can construct an induced path $P^i$ satisfying the induction hypothesis.
    \vspace{0.25cm}
    
    We now demonstrate how $C$ can construct a path $P^i$, satisfying all the desired assumptions, or capture the robber.
    Since we suppose the robber will not be caught, this will imply the path $P^i$ exists, 
    allowing us to proceed to the next stage.
    Let $\mathcal{R}$ be the connected component of $G-N[Q]$ containing the robber.
    By assumption $x$ has a neighbour in the connected component of $G-N[Q\setminus\{x\}]$ containing the robber.
    Then without loss of generality we take $v_1$ to be such a vertex,
    where $N(v_1)\cap \mathcal{R}\neq \emptyset$.
    Without loss of generality suppose that $C$ begins on $v_1$.
    Let $r_1$ be the starting vertex occupied by the robber, of course $r_1$ and $v_1$ are distinct, otherwise $C$ has caught the robber.
    Since $\mathcal{R}$ is a connected component this there is a $(v_1,r_1)$-path 
    all of whose internal vertices are in $\mathcal{R}$.
    Let $v_1,v_2,\dots, v_d,r_1$ be a shortest such path.

    As the following part of the game is played, we let $r_q$ be the location of the robber prior to the cop's $q^{\text{th}}$ moves in this stage.
    So the robber beings at $r_1$ and moves to $r_2$\dots etc.
    Since the robber might not move in a given turn, and the robber can revisit vertices,
    it is helpful to distinguish $r_q$ from $r_p$ even if they both correspond to the same vertex,
    hence the somewhat awkward  use of `location' rather than vertex in the previous sentence.
    We similarly distinguish between $r_q$ and a vertex $v_i$.
    We define the following linear order on these distinguished vertices
    \[
    v_1 \prec v_2 \prec \dots \prec v_d \prec r_1 \prec r_2 \prec \dots
    \]
    which we will use to define how the cop $C$ should move.

    From here we will let $u_q$ denote the cop $C$'s location
    before their $q^{\text{th}}$ move in this stage.
    So $u_1=v_1$.
    For the first $\lfloor \frac{t}{4}\rfloor$
    moves the cop $C$ always moves from 
    $u_q$
    to the vertex in $N(u_q)$ that is maximum under $\prec$ at the time they are moving.
    Notice here that as the robber moves, more locations are added to the order,
    so this stipulation that the cops moves the the largest $\prec$
    neighbour, evaluated when the cop moves, is necessary to the strategy being well defined.
    Since $v_1\dots v_dr_1$ is a $(v_1,r_1)$-path and $r_1r_2\dots r_q$ is a walk,
    and since $u_1=v_1$ this strategy is well defined, unless the cop has captured the robber, 
    since the robber's current index is always the greatest location in the order and this will be greater than the cop's current location.
    Since we suppose the robber is not caught, it follows that the cop's strategy is well defined.
    We note that this strategy means that every vertex $u_q$
    is a vertex in $\mathcal{R}$.
    Notice that as with $r_q$, in principle, it can be the case that $u_q$ and $u_p$ identify the same vertex.

    We now claim that $P^i = u_1\dots u_{\lfloor \frac{t}{4} \rfloor}$
    is an induced path of order $\lfloor \frac{t}{4} \rfloor$.
    To start, observe by definition that $P^i$ is a walk, so each vertex $u_i$ is in fact adjacent to $u_{i+1}$.

    Suppose $P^i$ is not an induced path.
    Since $P^i$ is a walk, this implies that for some integer $q$, the vertex $u_q$ has a neighbour in $\{u_1,\dots, u_{q-2}\}$.
    Here we include the possibility that $P^i$ contains a cycle of length $2$,
    that is for some $q$, $u_q = u_{q-2}$.
    Such a cycle of length $2$ corresponds to the cop $C$ leaving a vertex and returning to it on the next turn.
    Let $q$ be the smallest integer such that $u_q$ has a neighbour in $\{u_1,\dots, u_{q-2}\}$

    To start consider the case that $u_{q-2} = u_q$.
    Since the cop always moves to a vertex of largest order under $\prec$
    this implies that $u_{q-1}$ is the $\prec$ maximal neighbour of $u_{q}$
    and $u_q$ is the $\prec$ maximal neighbour of $u_{q-1}$.
    Since $u_q = u_{q-2}$ we can conclude that $u_q = r_i$ for some $i$.
    Take $i$ to be the largest index such that $u_q = r_i$.
    Since the robber is not captured by the cop, $r_i\neq r_q$ because if $u_q=r_q$, 
    then the cop $C$ has captured the robber.
    So there is a vertex $r_{i+1} \in N(u_q)$.
    We notice that, if $u_{q-1} = r_{i+1}$, then 
    $r_{i+2} \preceq u_q$ since $u_q$ is the $\prec$ maximal neighbour of $u_{q-1}$.
    But this contradicts $i$ being the maximum index such that $u_q=r_i$.
    Notice these facts could not have been effected by updating indices 
    since $C$ was on $u_{q-1}$ and $u_{q}= u_{q-2}$
    during the previous two turns while the robber remains uncaught.
    Hence, $r_{i+1}\neq u_{q-1}$.
    But this means that the maximum $\prec$ neighbour of $u_q = u_{q-2}$ 
    has updated since $C$ moved from $u_{q-2}$ to $u_{q-1}$.
    Thus, $r_q = r_{i+1}$ forcing $r_{q-1} = r_i = u_q$
    and $r_{q-2} = u_{q-1}$.
    But this implies $C$ caught the robber, a contradiction.
    Thus, $u_{q-2} \neq u_q$.
    Notice that we can conclude that $u_{q-1} \neq u_q$ follows by an even easier version of the argument.

    As a result, from the minimality of $q$, we conclude that
    $u_q$ is distinct from all vertices $u_p$ with $p<q$.
    Suppose that $p \leq q-2$ is the greatest integer such that $u_p \in N(u_q)$.
    Hence, $\mathcal{C} = u_pu_{p+1}\dots u_q$ is a cycle of length $3 \leq q-p+1  \leq \lfloor \frac{t}{4} \rfloor$.
    Consider the case that $q-p+1=3$.
    Then, $(u_q,u_{q-2})\in E(G)$.
    Since $v_1\dots v_d r_1$ is a geodesic path, this path is also induced,
    hence $u_q \notin \{v_1,\dots, v_d, r_1\}$.
    This implies $u_q = r_a$ for some $a\geq 2$ when $C$ form to $u_{q-1}$.
    Let $a$ be the largest index of a vertex $r_a=u_q$ before the cop moves to $u_{q-1}$.
    Since $u_q,u_{q-1} \in N(u_{p})$ the cop's strategy forces $r_a \prec u_{q-1} = r_b$ for some $b>a$ 
    when the cop moves to $u_{q-1}$.
    If $u_q\neq r_{b+1}$, then 
    the cop's strategy implies $u_q$ is the $\prec$ maximal neighbour of $u_{q-1}$,
    whereas $r_b$ will have $r_{b+1}$  as a neighbour while
    \[
    u_q=r_a \prec r_b \prec r_{b+1}.
    \]
   unless $r_{b+1}$ does not exist, 
   in which case $C$ has captured the robber, which we assume to not occur.
   If $u_q = r_{b+1}$ or $r_{b+1} \prec u_q$ when the cop moves from $u_{q-1}$ to $u_q$, then
   by our choice of $a$, 
   the robber
   moved to $u_q$ after the cop moved from $u_{q-2}$ to $u_{q-1}$.
   Hence, by moving from $u_{q-1}$ to $u_q$, the cop catches the robber, a contradiction.
   Otherwise, $q-p+1> 3$, in which case the minimality of $q$
   implies $\mathcal{C}$ is an induced cycle of length at least $4$
   and at most $\lfloor \frac{t}{4} \rfloor$
   contradicting that $G$ is $(C_4,\dots, C_t)$-free.
    From this we conclude that $P^i$ is
    an induced path as was claimed.

    Now we claim that there exists a neighbour of $u_{\lfloor \frac{t}{4} \rfloor}$ in the same component of 
    \[X := \mathcal{R}-(N[P^i\setminus \{u_{\lfloor \frac{t}{4} \rfloor}\}])\]
    as the robber, after the first robber turn following $C$ reaching $u_{\lfloor \frac{t}{4} \rfloor}$.
    Suppose otherwise for contradiction.
    Since $v_1\dots v_d r_1\dots r_{\lfloor \frac{t}{4} \rfloor}$ is a walk,
    the only way that the robber is in a component distinct from any neighbour of $u_{\lfloor \frac{t}{4} \rfloor}$
    is if the robber at some point moved through $(N[P^i\setminus \{u_{\lfloor \frac{t}{4} \rfloor}\}])$
    to a different component  of $X$.
    Thus, there is some vertex $r_q \notin V(P^i) \cup N(u_{\lfloor \frac{t}{4} \rfloor})$
    such that $r_q$ has a neighbour in $V(P^i)$.
    Suppose $q$ is the smallest integer such that $r_q$ has this property.
    Then, we can 
    observe that $r_q \notin N(u_{\lfloor \frac{t}{4} \rfloor})$,
    and we can notice that by the minimality of $q$ there is a path from $u_{\lfloor \frac{t}{4} \rfloor}$ to $r_{q-1}$
    in $X$.
    Letting $p$ be the largest integer such that $u_p \in N(r_q)$, and letting
    $L$ be a shortest $(u_{\lfloor \frac{t}{4}\rfloor},r_{q})$-path, whose internal vertices and edges are in $X$,
    it is clear that
    \[
    \mathcal{C} = u_pu_{p+1}\dots u_{\lfloor \frac{t}{4}\rfloor -1}L
    \]
    is an induced cycle.

    Since $G$ is $(C_4,\dots, C_t)$-free,
    this implies that $|\mathcal{C}|=3$ or $|\mathcal{C}|>t$.
    But if $|\mathcal{C}|=3$, then $r_q\in N(u_{\lfloor \frac{t}{4}\rfloor})$ a contradiction.
    So, $|\mathcal{C}|>3$ and as a result $|\mathcal{C}|>t$.
    Notice that if $u_{\lfloor \frac{t}{4}\rfloor} = r_a$ for some $a$,
    then the minimality of $q$ implies that there exists a path  in $X$,
    namely some subpath of the $r_ar_{a+1}\dots r_q$ walk,
    which has length at most $q-a \leq \lfloor \frac{t}{4}\rfloor-1$.
    Since $L$ is a shortest such path this implies $|L| \leq \lfloor \frac{t}{4}\rfloor$.
    It follows that 
    \[
    t < |\mathcal{C}| \leq |L| + \Bigg(\Big\lfloor \frac{t}{4}\Big\rfloor - p+1\Bigg) \leq 2\Big\lfloor \frac{t}{4}\Big\rfloor+1 < t
    \]
    a contradiction.
    It must be the case then that for all $a$, $u_{\lfloor \frac{t}{4}\rfloor} \neq r_a$.
    
    Given the cop $C$'s rule for moving this implies $P^i$ is a subpath of $v_1\dots v_d$.
    Suppose $u_{\lfloor \frac{t}{4}\rfloor}=v_a$, then $u_p= v_b$ where $b < a$.
    As before we note that the minimality of $q$ implies 
    there is a path from $u_{\lfloor \frac{t}{4} \rfloor}$ to $r_{q-1}$
    in $X$.
    As before let $L$ be a shortest $(u_{\lfloor \frac{t}{4}\rfloor},r_q)$-path whose internal vertices and edges 
    are in $X$.
    Again, \[
    \mathcal{C} = u_pu_{p+1}\dots u_{\lfloor \frac{t}{4}\rfloor-1}L
    \]
    is an induced cycle
    and $|\mathcal{C}|>3$ forcing $|\mathcal{C}|> t$.
    Notice that 
    there exists a path  in $X$,
    namely some subpath of the $v_{a+1}\dots v_dr_1r_{1}\dots r_q$ walk,
    which has length at most $(d-a)+q \leq (d-a)+\lfloor \frac{t}{4}\rfloor$.
    Also recall that the path $v_1\dots v_d r_1$ is a geodesic path in all of $\mathcal{R}$.
    Hence, the distance $\dist_\mathcal{R}(r_1,u_p) = d-b+1$.
    Because the robber starts on $r_1$ and at $r_q$ the robber is adjacent to $u_p$ this implies that
    that $d-b\leq q \leq \lfloor \frac{t}{4}\rfloor$.
    From this we can deduce that 
    $d-b \leq \lfloor\frac{t}{4}\rfloor$ implying that
    \[
    t< |\mathcal{C}| \leq (a-b)+(d-a+1) + q  = (d-b+1)+q \leq   2\Big\lfloor \frac{t}{4}\Big\rfloor+1 < t
    \]
    also a contradiction.
    Thus, we conclude that there is in fact a neighbour of $u_{\lfloor \frac{t}{4}\rfloor}$ in the same component of 
    $X$
    as the robber after $P^i$ is finished being defined.
    \hfill $\diamond$
    \vspace{0.5cm}

    \vspace{0.25cm}
    \noindent\underline{Claim.2}: The cop $C$ can guard $N[P^i]$ from the first turn after they reach $u_{\lfloor \frac{t}{4} \rfloor}$.
    \vspace{0.25cm}
    
    The final claim in Phase~1 is
    that the cop can guard $N[P^i]$ from the first turn after they reach $u_{\lfloor \frac{t}{4} \rfloor}$.
    Suppose then the robber has made exactly $1$ move after the cop reached $u_{\lfloor \frac{t}{4} \rfloor}$
    and it is the cop's turn.
    Observe that the robber could not
    have moved into $N(P^i)$ without having moved into $N(C) = N(u_{\lfloor \frac{t}{4} \rfloor})$
    by the same arguments as the previous claim.
    Let $r$ denote the current location of the robber, then $r \in \mathcal{R} - N[P^i]$.
    Let $Y := \mathcal{R} - N[P^i]$.

    If $\dist_{Y}(r,N(P^i))\geq \lfloor \frac{t}{2}- \lfloor \frac{t}{4}\rfloor-1\rfloor$,
    then by Lemma~\ref{Lemma: 1-guard} (2), the cop $C$ can guard $N[P^i]$
    indefinitely.
    Suppose then that $\dist_{Y}(r,N(P^i))< \lfloor \frac{t}{2}- \lfloor \frac{t}{4}\rfloor-1\rfloor \leq \frac{t}{4}$.
    If $\dist_{Y}(r,N(u_{\lfloor \frac{t}{4} \rfloor}))\leq \lfloor \frac{t}{2}- \lfloor \frac{t}{4}\rfloor-1\rfloor$,
    then the cop $C$ moves to a vertex $z \in N(u_{\lfloor \frac{t}{4} \rfloor})$
    such that 
    $
    \dist_{Y}(r,z)\leq \lfloor \frac{t}{2}- \lfloor \frac{t}{4}\rfloor-1\rfloor
    $
    and by 
    Lemma~\ref{Lemma: 1-guard} (1), the cop $C$ can guard $N[P^i]$
    indefinitely.
    Otherwise,
    \[
    \dist_{Y}(r,N(P^i))< \dist_{Y}(r,N(u_{\lfloor \frac{t}{4} \rfloor})).
    \]
    Let $w \in N(u_p)$ be a vertex with minimum distance, in $Y$, to $r$ in $N(P^i)$,
    and let $z$ be the vertex with minimum distance, in $Y$, to $r$ in $N(u_{\lfloor \frac{t}{4} \rfloor})$.
    Without loss of generality choose $p$ as large as possible.
    It is the cop's turn and they move from $u_{\lfloor \frac{t}{4} \rfloor}$ to $z$.
    Now it is the robber's turn and they have not moved yet.

    If $u_{\lfloor \frac{t}{4} \rfloor} = r_a$ for some $a$, then 
    $r_a\dots r_{\lfloor \frac{t}{4} \rfloor-1} r$ is a walk, 
    whose only vertices in $N[P^i]$ are $r_a$ and $r_{a+1}$
    by minimality of $q$.
    Hence, $\dist_{Y}(r,N(u_{\lfloor \frac{t}{4} \rfloor})) \leq \lfloor \frac{t}{4} \rfloor$.
    Hence, if $(w,z)\notin E(G)$, then 
    $w,z$ are non-adjacent vertices in $N(P^i)$ and by the triangle-inequality 
    \[
    \dist_Y(w,z) \leq \dist_Y(z,r)+\dist_Y(r,w) \leq \frac{t}{4} + \frac{t}{4} < t-\lfloor \frac{t}{4} \rfloor
    \]
    contradicting Lemma~\ref{Lemma: Common neighbourhood}.
    Thus, if $u_{\lfloor \frac{t}{4} \rfloor} = r_a$ for any $a$, then $(w,z)\in E(G)$.

    If $u_{\lfloor \frac{t}{4} \rfloor} \neq r_a$ for some $a$, then 
    it follows that $P^i = v_1\dots v_\ell \dots v_{\lfloor \frac{t}{4} \rfloor}$ where $u_{\lfloor \frac{t}{4} \rfloor} = v_{\lfloor \frac{t}{4} \rfloor}$.
    Here $u_{p} = v_p$.
    In this case, since $v_1 \dots v_{d}r_1$ is a shortest path in all of $\mathcal{R}$
    we notice that
    \begin{align*}
         \frac{t}{4}  &> \dist_{Y}(r,w) \\
        &\geq \dist_\mathcal{R}(w,r) \\
        & \geq \dist_\mathcal{R}(u_p,r)-1.
    \end{align*}
    From here we note that
    since $v_1\dots v_dr_1$ was a geodesic in $\mathcal{R}$, we have
    $\dist_\mathcal{R}(u_p,r_1) = d-p+1$.
    Meanwhile $\dist_\mathcal{R}(r_1,r) \leq \lfloor \frac{t}{4}\rfloor$.
    Since $\dist_\mathcal{R}(u_p,r) \leq \frac{t}{4}$
    we conclude by the triangle inequality that
    \[
    d-p+1 = \dist_\mathcal{R}(u_p,r_1) \leq \dist_\mathcal{R}(u_p,r)+\dist_\mathcal{R}(r_1,r) \leq \frac{t}{4} + \lfloor \frac{t}{4}\rfloor \leq \frac{t}{2}.
    \]
    Thus, $d \leq \frac{t}{2} + p-1$.

    From here we notice that, again by the triangle inequality, and the fact that the robber never entered $N(P^i)$
    while the cop was travelling to $u_{\lfloor \frac{t}{4} \rfloor}$  by the minimality
    of $q$, that
    \begin{align*}
        \dist_Y(z,r) & < \dist_\mathcal{R}(u_{\lfloor \frac{t}{4}\rfloor } ,r_1) + \dist_\mathcal{R}(r_1,r) \\
        &\leq \Big(d-\lfloor \frac{t}{4} \rfloor +1 \Big) + \lfloor \frac{t}{4} \rfloor \\
        & = d+1 \\
        & \leq \frac{t}{2} + p.
    \end{align*}
    Hence, $\dist_Y(z,r) \leq \frac{t}{2} + p-1$.

    It follows that 
    \begin{align*}
        \dist_Y(w,z) &\leq \dist_Y(w,r)+\dist_Y(r,z) \\
        &< \frac{t}{4} + \Big( \frac{t}{2} + p-1\Big)\\
        & = \frac{3t}{4} + p -1
    \end{align*}
    Letting $\mathcal{P} = u_{p}\dots u_{\lfloor \frac{t}{4}\rfloor}$, we note that $\mathcal{P}$
    is an induced path of order $\lfloor \frac{t}{4}\rfloor-p+1$.
    Now,
    if $(w,z)\notin E(G)$, then 
    applying Lemma~\ref{Lemma: Common neighbourhood}
    we observe that 
    \begin{align*}
        \frac{3t}{4} + p -1 & > \dist_Y(w,z)\\
        &\geq \dist_{\mathcal{R}-N[\mathcal{P}]}(w,z) \\
        & \geq t-|V(\mathcal{P})|
        & = t - \Big(\lfloor \frac{t}{4}\rfloor-p+1\Big) \\
        & \geq \lceil \frac{3t}{4}\rceil +p-1
    \end{align*}
    a contradiction.
    Thus, $(w,z)\in E(G)$ regardless of whether $u_{i\lfloor \frac{t}{4} \rfloor} = r_a$ for some $a$ or not.
    Since $w$ was selected without loss of generality with respect to vertices in $N(P^i)$
    that minimize distance in $Y$ to $r$,
    we conclude that all such vertices are neighbours of $z$.

    Now, the it is the robber's turn, the cop occupies $z$ and there is an edge $(w,z)\in E(G)$.
    The robber moves from $r$ to some vertex $R$.
    These might be the same vertex.
    If $R \in N(P^i)$, then $(R,z)\in E(G)$ and the cop on $z$ can catch the robber on $R$,
    since we have shown any vertex in $N(P^i)$
    that minimize distance in $Y$ to $r$ is a neighbour of $z$.
    So we suppose that $R\notin N(P^i)$.
    We notice that it is the cop's turn to move and
    \[
    \dist_Y(w,R) \leq \dist_Y(w,r) + 1 \leq \lfloor \frac{t}{2}-\lfloor\frac{t}{4}\rfloor-1\rfloor
    \]
    by our assumption that $\dist_Y(w,r) < \lfloor \frac{t}{2}-\lfloor\frac{t}{4}\rfloor-1\rfloor$.
    Then $C$ can move from $z$ to $w$, which is possible since $(w,z)\in E(G)$,
    and the cop is then in state (1).
    This completes the proof of the claim.
    \hfill $\diamond$
    \vspace{0.5cm}

    Thus, in summary, we have constructed an induced path $P^i$ satisfying all the assumptions of the induction hypothesis.
    Moreover, we have positioned one cop $C$ on $P^i$ so that $N[P^i]$ is guarded indefinitely.
    Therefore, we have satisfied the induction hypothesis.
    This allows us to assume that after finite time
    for all $1\leq j \leq \Big\lceil \frac{k-2}{\lfloor \frac{t}{4}\rfloor} \Big\rceil$, there is a distinct induced path $P^j$ of order $\lfloor\frac{t}{4} \rfloor$,
    such that $N[P^j]$ is guarded by $1$ cop who has no other responsibilities.
    All paths $P^j$ are vertex disjoint.
    Also, the union of all vertices $\cup_{j=1}^{i-1} V(P^j)$ induces a path, call it $Q$,
    while there is an endpoint $x$ of $Q$ which has a neighbour in the component of $G-N[Q\setminus \{x\}]$ containing the robber.
    From this state we proceed to Phase~2.

    \vspace{0.25cm}
    \noindent\underline{Phase~2:} 
    \vspace{0.25cm}

    Label our cops $C^i$ for all $1\leq i \leq \lceil \frac{k-2}{\lfloor\frac{t}{4}\rfloor} \rceil$.
    In Phase~2 we argue that,
    given an induced path $Q = q_1\dots q_M$
    such that 
    \begin{enumerate}[(i)]
        \item $M\geq \lfloor\frac{t}{4}\rfloor \lceil \frac{k-2}{\lfloor\frac{t}{4}\rfloor} \rceil$, and 
        \item for all $1\leq i \leq \lceil \frac{k-2}{\lfloor\frac{t}{4}\rfloor} \rceil$,
        the cop $C^i$ guards the closed neighbourhood of
        the subpath $Q_i = q_{(i-1)\lfloor\frac{t}{4}\rfloor+1}\dots q_{i\lfloor\frac{t}{4}\rfloor}$, and
        \item $q_1$ has a neighbour in the same connected component of $G-N[Q\setminus \{q_1\}]$
        as the robber,
    \end{enumerate}
    our set of cops can either catch the robber, or redeploy 
    to a new induced path $Q' = q'_1\dots q'_{M+1}$ so that 
    \begin{enumerate}[(I)]
        \item for all $j\geq 1$, $q_j = q'_{j+1}$, and 
        \item for all $1\leq i \leq \lceil \frac{k-2}{\lfloor\frac{t}{4}\rfloor} \rceil$,
        the cop $C^i$ guards the closed neighbourhood of
        the subpath $Q'_i = q'_{(i-1)\lfloor\frac{t}{4}\rfloor+1}\dots q'_{i\lfloor\frac{t}{4}\rfloor}$, and
        \item $q'_1$ has a neighbour in the same connected component of $G-N[Q'\setminus \{q'_1\}]$
        as the robber, and 
        \item the robber never entered $N[Q'\setminus\{q'_1\}]$
        during the time it took the cops to redeploy from $Q$ to $Q'$.
    \end{enumerate}
    Of course the conclusion of Phase~1 leaves us in a position 
    that satisfies conditions (i)-(iii).
    Meanwhile condition (I)-(IV) of the redeployment means the resulting path also satisfies (i)-(iii).
    Hence, we can recursively reapply this claim, until either the robber
    is captured or the entire graph is part of the closed neighbourhood of induced path we are constructing.
    Since the graph $G$ is finite, the latter implies the former since condition (IV) forces that 
    the robber would be caught if its location is part of the closed neighbourhood of $Q'$.

    Suppose then there exists a path $Q$ satisfying (i)-(iii).
    Let $\mathcal{R}$ be the connected component of  $G-N[Q]$  containing the robber.
    As before, if $\mathcal{R}$ is empty, then the robber is captured.
    Suppose then it is not empty.
    By (iii) and the definition of $\mathcal{R}$ there exists a vertex $x \in N(q_1)\cap N(\mathcal{R})$
    such that $N(x)\cap V(Q) = \{q_1\}$.
    Let $x$ be such a vertex.
    Let $q'_1 = x$ and for all $j\geq 1$ let $q_j = q'_{j+1}$, let  $Q' = q'_1\dots q'_{M+1}$.
    So (I) and (III) are satisfied.
    Since $N(x)\cap V(Q) = \{q_1\}$ we note that $Q'$ is an induced path, as desired.
    Begin the game state from this position.

    We begin by noting that the robber cannot enter 
    $
    N[\{q_{k-2},\dots, q_M\}]\setminus N[\{q_{1},\dots, q_{k-3}\}]
    $
    while the cops have been guarding $N[\{q_{1},\dots, q_{k-3}\}]$ since the beginning of the game state.
    Suppose for contradiction that the robber where to enter $N[\{q_{k-2},\dots, q_M\}]\setminus N[\{q_{1},\dots, q_{k-3}\}]$.
    Then there exists a vertex $r$ in $\mathcal{R}$ and a vertex $z \in N(\{q_{k-2},\dots, q_M\})\setminus N(\{q_{1},\dots, q_{k-3}\})$ such that $(r,z)\in E(G)$.
    Without loss of generality, let $p$ be the smallest integer such that $z \in N(q_p)$.
    Then $p\geq k-2$.
    Notice that $r\notin N(Q)$, and by (iii) there is a $(q_1,r)$-path of order at least $3$ all of whose
    internal vertices are in $\mathcal{R}$ or in $N(q_1)\setminus N[Q\setminus \{q_1\}]$.
    Call $P$ such a path.
    Then, $\mathcal{C} = Pzq_pq_{p-1}\dots q_{2}$ is a cycle,
    and the only possible chords in $\mathcal{C}$ are incident to the internal vertices of $P$.
    Since $p\geq k-2$ this implies that there is a hole of length at least $k$ in $\mathcal{C}$.
    This contradicts that $G$ is $(C_k,C_{k+1},\dots)$-free.
    Thus, we conclude that the robber cannot enter $N[\{q_{k-2},\dots, q_M\}]\setminus N[\{q_{1},\dots, q_{k-3}\}\}]$
    as desired.

    For each $i$ let $Q'_i$ be defined as in (II).
    We note that the guarding of paths $Q_i$ leaves exactly $1$ cop free.
    Call this cop $C^0$.
    The cop's strategy begins by having $C^0$ walk to $q'_1$.
    They will stand there for the remainder of the redeployment, preventing the robber entering $N[q'_1]$
    from this point onward.
    
    Next, we claim that for $i\geq 1$, if cop $C^{i-1}$ is guarding $N[Q_{i-1}]$
    the cop $C^i$, which begins guarding $N[Q_i]$, can move to guard $N[Q'_i]$
    so that $N[V(Q_i)\cap V(Q'_i)]$ is guarded in all intermediate steps.
    For $i=1$ this requires $C^0$ standing on $q'_1$ as we have outlined.
    We will the apply this strategy to all cops $C^i$, for $i\geq 1$, in decreasing order, one at a time.
    This allows,
    each cop $C^i$, for $i\geq 1$, to switch from guarding $Q_i$ to $Q'_i$ while ensuring that guarding $N[\{q_{1},\dots, q_{k-3}\}]$ is guarded at all times since the beginning of the game state.
    Because of that assumption, we can suppose the robber never enters the $N[\{q_{k-2},\dots, q_M\}]\setminus N[\{q_{1},\dots, q_{k-3}\}]$
    by the previous paragraph's argument.
    Implicit here is the easy
    observation that for all $j$,
    \[
    V(Q_j)\cap V(Q'_j) = V( Q_j)\setminus \{q_{j\lfloor\frac{t}{4}\rfloor}\} =  V( Q'_j)\setminus \{q'_{(j-1)\lfloor\frac{t}{4}\rfloor+1}\}.
    \]
    Thus, these claims will imply (II) and (IV).
    
    All guarding up to this point is done using the strategies in Lemma~\ref{Lemma: 1-guard},
    which involves the cop being in one of the two prescribed states (1) or (2).
    We can suppose then without loss of generality then that $C^i$ is in state (1) or state (2) for $N[Q_i]$,
    as long as we then show $C^i$ will guard $N[Q'_i]$ by achieving state (1) or state (2) from Lemma~\ref{Lemma: 1-guard}
    for $N[Q'_i]$.

    Let $H = G - N[Q']$.
    Since the robber begins in $\mathcal{R}$ 
    and we will suppose the robber has never entered $N(Q)$ between the current state and the beginning of the game state
    the robber is restricted to moving in $H$ until the current state.
    Also, once $C^0$ reaches $q'_1$ the robber will never
    enter $N[q'_1]$ for the rest of the redeployment.
    We then prove the robber cannot leave $H$ in the current state without being caught.
    Observe that the cops' strategy only become easier if distances outside 
    the guarded set increase.
    Hence, cops in state (1) or (2) for a large subgraph, 
    can trivially transfer to guarding in state (1) or state (2)
    for a subgraph with larger distances for the robber.
        Thus, when measuring distance for the purpose of Lemma~\ref{Lemma: 1-guard}
    we can always do so in $H$, and we can always assume our cops guarding $N[Q_i]$
    are in state (1) or (2) for guarding $N[Q_i]$ using distances in $H$.

    \vspace{0.25cm}
    \noindent\underline{Case~1:} The cop $C^i$ begins in state (1) for guarding $N[Q_i]$.
    \vspace{0.25cm}
    
    To begin, suppose that $C^i$ is 
    in state (1) for the purposes of guarding $N[Q_i]$.
    Let $u$ be the vertex occupied by $C^i$.
    Without loss of generality let $q_j \in N(u) \cap V(Q_i)$.
    Since $C^i$ is in state (1) it is the robbers turn. 
    Suppose the robber move to a vertex
    $R$.

    Consider first the case where $\dist_H(R,N(Q'_i))\leq \lfloor \frac{t}{2}- \lfloor \frac{t}{4}\rfloor-1\rfloor$.
    Since $C^i$ was in state (1), $R \in N(Q_i)$ implies the cop $C^i$ catches the robber.
    Together with our assumption that $C^{i-1}$ is guarding $N[q_{(i-1)\lfloor\frac{t}{4}\rfloor}]$
    implies that so suppose $R\notin N(Q'_i)$.
    Let $w \in N(Q'_i)$ such that $\dist_H(R,w)\leq \lfloor \frac{t}{2}- \lfloor \frac{t}{4}\rfloor-1\rfloor$.
    If $(u,w)\in E(G)$, then $C^i$ moves from $u$ to $w$ and $C^i$ is now in state (1) for guarding $N[Q'_i]$.
    Suppose then that $(u,w)\not\in E(G)$.
    Then, since $u,w$ are non-adjacent neighbours of $Q_i\cup Q'_i$, an induced path of order $\lfloor \frac{t}{4} \rfloor+1$.
    By Lemma~\ref{Lemma: Common neighbourhood} implies that 
    \begin{align*}
        t - \Big(\lfloor \frac{t}{4} \rfloor+1\Big) & = \lceil \frac{3t}{4} \rceil - 1\\
        & \leq \dist_{G-N[Q_i\cup Q'_{i}]}(u,w) \\
        &\leq \dist_H(u,w) \\
        &\leq \dist_H(u,R) + \dist_H(R,w) \\
        & \leq \dist_H(u,r)+1 + \dist_H(R,w) \\
        & \leq 2 \Big(\frac{t}{2}-\lfloor \frac{t}{4}\rfloor-1 \Big) + 1 \\
        & = t -2\lfloor\frac{t}{4}\rfloor -1
    \end{align*}
    a contradiction for $t\geq 9$.

    Now consider the case where $\dist_H(R,N(Q'_i))> \lfloor \frac{t}{2}- \lfloor\frac{t}{4}\rfloor-1\rfloor$.
    If  $\dist_H(R,N(Q'_i))\geq (\lfloor \frac{t}{2}- \lfloor\frac{t}{4}\rfloor-1\rfloor)+2$,
    then $C^i$ moves to $q_j$.
    On their next turn, if $q_j \in V(Q'_i)$, then $C^i$ stays still.
    Otherwise, $q_j = q_{i\lfloor \frac{t}{4} \rfloor}$, and in this case, or $C^i$
    moves from $q_{i\lfloor \frac{t}{4} \rfloor}$ to $q_{i\lfloor \frac{t}{4} \rfloor-1}$.
    Since $\dist_H(R,N(Q'_i))\geq (\lfloor \frac{t}{2}- \lfloor\frac{t}{4}\rfloor-1\rfloor)+2$,
    regardless of how the robber responds, say they end on a vertex $W$ following the reply to $C^i$ moving to $q_j\in V(Q'_i)$,
    we observe that
    \[
    \dist_H(W,N(Q'_i))\geq \dist_H(R,N(Q'_i))-2 \geq \lfloor \frac{t}{2}- \lfloor\frac{t}{4}\rfloor-1\rfloor
    \]
    while $C^i$ occupies a vertex of $Q'_i$.
    Thus, $C^i$ is in state (2) for the purpose of guarding $N[Q'_i]$ since it is the cops turn, 
    following the robber's move to $W$.
    Suppose then that $\dist_H(R,N(Q'_i))= \lfloor \frac{t}{2}-\lfloor \frac{t}{4}\rfloor \rfloor$.

    Since $\dist_H(R,N(Q'_i))= \lfloor \frac{t}{2}-\lfloor \frac{t}{4}\rfloor \rfloor$ there exists 
    a vertex $w \in N(Q'_i)$ such that $\dist_H(R,w) = \lfloor \frac{t}{2}-\lfloor \frac{t}{4}\rfloor \rfloor$.
    If $(u,w)\notin E(G)$, then, as above, 
    $u,w$ are non-adjacent neighbours of $Q_i\cup Q'_i$, an induced path of order $\lfloor \frac{t}{4} \rfloor+1$, from which
    Lemma~\ref{Lemma: Common neighbourhood} implies that 
    \begin{align*}
        t - \Big(\lfloor \frac{t}{4} \rfloor+1\Big) & = \lceil \frac{3t}{4}\rceil -1\\
    \end{align*}
    \begin{align*}
        & \leq \dist_{G-N[Q_i\cup Q'_{i}]}(u,w) \\
        &\leq \dist_H(u,w) \\
        &\leq \dist_H(u,R) + \dist_H(R,w) \\
        & \leq \dist_H(u,r)+1 + \dist_H(R,w) \\
        & \leq  2 \Big( \frac{t}{2}-\lfloor \frac{t}{4}\rfloor-1\Big) + 2\\
        & = t - 2\lfloor\frac{t}{4}\rfloor.
    \end{align*}
    a contradiction for $t\geq 13$. So $(u,w)\in E(G)$.
    Thus, the cop $C^i$ moves from $u$ to $w$ in response to the robber's move from $r$ to $R$.
    
    Suppose the robber $R$ next moves to a vertex $X$ with $\dist_H(X,N(Q'_i))=  \lfloor \frac{t}{2}- \lfloor\frac{t}{4}\rfloor-1\rfloor$.
    If $\dist_H(X,w) =  \lfloor \frac{t}{2}- \lfloor\frac{t}{4}\rfloor-1\rfloor$, then the cop passes and on the robber's next turn $C^i$ is
    in state (1) for guarding $N[Q'_i]$.
    Otherwise, there exists a vertex $x\neq w$ such that $x \in N(Q'_i)$ and  $\dist_H(X,x) =  \lfloor \frac{t}{2}- \lfloor\frac{t}{4}\rfloor-1\rfloor$.
    If $(w,x)\in E(G)$ the cop $C^i$ moves to $x$ and it is now in state (1) for guarding $N[Q'_i]$.
    So consider $(w,x)\not\in E(G)$.
    Then,
    $x,w$ are non-adjacent neighbours of $Q'_i$, an induced path of order $\lfloor \frac{t}{4} \rfloor$, from which
    Lemma~\ref{Lemma: Common neighbourhood} implies that 
    \begin{align*}
        t - \lfloor \frac{t}{4} \rfloor & = \lceil\frac{3t}{4}\rceil\\
        &\leq \dist_{G-N[Q'_{i}]}(x,w) \\
        &\leq \dist_H(x,w) \\
        &\leq \dist_H(x,X) + \dist_H(X,w) \\
        & \leq \dist_H(x,X)+ \dist_H(R,w) +1\\
        & \leq  2 \Big( \frac{t}{2}-\lfloor \frac{t}{4}\rfloor-1 \Big) + 2 \\
        & = t - 2\lfloor\frac{t}{4}\rfloor.
    \end{align*}
    a contradiction for $t\geq 9$.

    Suppose then the robber moves to a vertex $Y$ with $\dist_H(Y,N(Q'_i))\geq \lfloor \frac{t}{2}-\lfloor \frac{t}{4}\rfloor \rfloor $.
    Then the cop $C^i$ moves from $w$ to any neighbour of $w$ in $Q'_i$.
    No matter how the robber responds, the cop $C^i$ will be in state (2) at the start of their next turn.

    Thus, no matter how the robber moves, the cop $C^i$ beginning in state (1) for guarding $N[Q-i]$
    can redeploy to state (1) or state (2) for guarding $N[Q'_i]$ without the robber entering $N[Q'_i]$.
    \hfill $\diamond$
    \vspace{0.5cm}

    \vspace{0.25cm}
    \noindent\underline{Case~2:} The cop $C^i$ begins in state (2) for guarding $N[Q_i]$.
    \vspace{0.25cm}
    
    Next, suppose that $C^i$ is in state (2) for the purposes of guarding $N[Q_i]$.
    Let $q_N$ be the vertex occupied by $C^i$.
    Then $N\leq i\lfloor\frac{t}{4}\rfloor$.
    Note $q_N \in V(Q_i)$, we suppose without loss of generality $q_N = q_{i\lfloor\frac{t}{4}\rfloor}$.
    If this is not true and $C$ begins in state (2) for the purpose of guarding $N[Q'_i]$, then
    we are done.
    Otherwise, the robber begins within distance of $\lfloor \frac{t}{2}- \lfloor\frac{t}{4}\rfloor-1\rfloor$
    of $N(q_{(i-1)\lfloor\frac{t}{4}\rfloor})$ 
        but outside distance $\lfloor \frac{t}{2}- \lfloor\frac{t}{4}\rfloor-1\rfloor$
        of any other vertex in $Q'_i$
        and same argument can be applied. 
    As the game progresses let $R_d$ be the location of the robber after $d$ moves have been made by the robber 
    since the 
    start of the case. 
    So the robber starts at $R_0$.

    Since $C^i$ begins in state (2) it is the cops turn to start.
    For all $d\geq 0$
    let $S_d \subseteq N[Q'_i]$ such that
    $s \in S_d$ if and only if $\dist_H(R_d,s)\leq \lfloor \frac{t}{2}- \lfloor\frac{t}{4}\rfloor-1\rfloor$.

    The cop $C^i$ proceeds via the following strategy.
    Beginning at $q_{i\lfloor\frac{t}{4}\rfloor}$ the cop $C^i$ walks along 
    $Q'_i$ to $q_{(i-1)\lfloor\frac{t}{4}\rfloor}$, not paying attention to the robber,
    unless the cop is interrupted by one of the following situations, or the cop $C^i$ 
    reaches state (1) or state (2) for the purpose of guarding $N[Q'_i]$.
    Each of these interruptions will lead to $C^i$ being 
    able to reach state (1) or state (2) for the purpose of guarding $N[Q'_i]$,
    so we can suppose that whenever an interruption occurs, it is the first interruption.
    Hence, after $d$ robber moves, the cop $C^i$ will be occupying vertex $q_{i\lfloor\frac{t}{4} \rfloor-d}$.

    The first interruption for the cop $C^i$ is if $S_d = \emptyset$.
    If $S_d = \emptyset$, then the cop moves to a vertex of $Q'_i$, which is always possible
    since $C^i$ is either on a vertex of $Q'_i$ or $d=0$.
    Since $S_d = \emptyset$ we observe that $\dist_H(R_{d+1},N(Q'_i))\geq \lfloor \frac{t}{2}- \lfloor\frac{t}{4}\rfloor-1\rfloor$.
    Thus, on cop's next turn, $C^i$ is in state (2) for the purposes of guarding $N[Q'_i]$.

    The second interruption for the cop $C^i$ is if
    $q_{i\lfloor\frac{t}{4} \rfloor-d} \in N(S_d)$.
    In this case $C^i$ moves to a vertex in $S_d$
    and the cop $C^i$ is in state (1) for the purposes of guarding $N[Q'_i]$.

    We claim that there must exist an integer $0 \leq d\leq \lfloor \frac{t}{4} \rfloor$ such that 
    $C^i$ is interrupted while standing on vertex $q_{i\lfloor\frac{t}{4}\rfloor-d}$.
    Suppose not for the sake of contradiction.

    Since the first interruption never occurs it must be true that $S_d\neq \emptyset$
    for all $0 \leq d\leq \lfloor \frac{t}{4} \rfloor$.
    Next, for each $x \in N(Q'_i)$ we let $f(x) = p - (i-1)\lfloor \frac{t}{4}\rfloor$ where $q_p$ is the vertex of largest index in
    $V(Q'_i)\cap N(x)$.
    We claim that for all $d$, if $x \in S_d$, then $f(x)\leq \lfloor \frac{t}{4}\rfloor -d$.
    
    Suppose not for the sake of contradiction and suppose that $d$ is the smallest integers such that 
    there exists a vertex $x \in S_d$ with $f(x)> \lfloor \frac{t}{4}\rfloor -d$.
    Let $x$ be such a vertex.
    Notice that by the definition of $f$, $f(x) < \lfloor \frac{t}{4}\rfloor$ for all $x$,
    hence $d\geq 1$.

    We claim that $x \notin S_{d-1}$.
    Suppose so for contradiction.
    Then $x \in S_{d-1}$ and by the minimality of $d$,
    $f(x) \leq \lfloor \frac{t}{4}\rfloor -d+1$.
    Hence, $f(x) = \lfloor \frac{t}{4}\rfloor -d+1$
    implying that $x \in N(q_{i\lfloor \frac{t}{4}\rfloor -d+1})$.
    But before making their move $d$, the cop $C^i$ occupies 
    $q_{i\lfloor \frac{t}{4}\rfloor -d+1}$,
    implying the cop $C^i$ is adjacent to $x \in S_{d-1}$.
    Thus, 
    cop $C^i$ is interrupted while standing on vertex 
    $q_{i\lfloor \frac{t}{4}\rfloor -d+1}$ a contradiction.
    We conclude that $x \notin S_{d-1}$.

    Since $x\notin S_{d-1}$ and $S_{d-1}\neq \emptyset$
    there exists a vertex $y\neq x$ such that $y \in S_{d-1}$.
    By the minimality of $d$, $f(y) \leq \lfloor \frac{t}{4}\rfloor -d+1$ .
    Notice that if $f(y) = \lfloor \frac{t}{4}\rfloor -d+1$, then this would interrupt 
    the cop $C^i$, since the cop's strategy would force a move to $y$ rather than $q_{ \lfloor \frac{t}{4}\rfloor -d}$.
    Since this does not occur by assumption $f(y) \leq \lfloor \frac{t}{4}\rfloor -d$.
    Since 
    $y \in S_{d-1}$ and $x \in S_{d}\setminus S_{d-1}$
    we note that 
    \begin{align*}
        \dist_H(R_{d-1},y)\leq \lfloor \frac{t}{2}- \lfloor\frac{t}{4}\rfloor-1\rfloor & &\text{and }& &
        \dist_H(R_{d},x)= \lfloor \frac{t}{2}- \lfloor\frac{t}{4}\rfloor-1\rfloor.
    \end{align*}
    If $(x,y)\notin E(G)$, then Lemma~\ref{Lemma: Common neighbourhood}
    implies 
    \begin{align*}
        t - \lfloor \frac{t}{4} \rfloor & = \lceil \frac{3t}{4} \rceil\\
        & \leq \dist_{G-N[Q'_{i}]}(x,y) \\
        &\leq \dist_H(x,y) \\
        &\leq \dist_H(x,R_{d}) + \dist_H(R_d,y) \\
        & \leq \dist_H(x,R_{d}) + \dist_H(R_{d-1},y) + 1 \\
    \end{align*}
    \begin{align*}
        & \leq 2 \Big( \frac{t}{2}-\lfloor \frac{t}{4}\rfloor-1 \Big) + 1 \\
        &= t-2\lfloor\frac{t}{4}\rfloor -1
    \end{align*}
    a contradiction for $t\geq 5$.
    Otherwise, $(x,y)\in E(G)$.
    Recalling that since $G$ is $(C_4,\dots, C_t)$-free $N(x) \cap V(Q'_i)$ and $N(y) \cap V(Q'_i)$
    form subintervals of $Q'_i$. 
    Observe that these intervals must intersect,
    otherwise, there is a hole of order at most $|V(Q'_i)|+2 \leq t$
    formed by a subinterval of $Q'_i$ and $xy$.
    From this it follows that $\{q_{f(y)+(i-1)\lfloor\frac{t}{4}\rfloor},q_{f(y)+1+(i-1)\lfloor\frac{t}{4}\rfloor}\dots, q_{f(x)+(i-1)\lfloor\frac{t}{4}\rfloor}\} \subseteq N(x)$.
    But $C^i$ occupies vertex $q_{i\lfloor\frac{t}{4}\rfloor - d}$, the robber occupies vertex $R_d$,
    and $f(y)+(i-1)\lfloor\frac{t}{4}\rfloor \leq i\lfloor\frac{t}{4}\rfloor - d < f(x)+(i-1)\lfloor\frac{t}{4}\rfloor$.
    Therefore, $C^i$ is adjacent to $x$, implying the cop's strategy is interrupted by a move to $x\in S_d$, a contradiction.
    We conclude that for all $d$, and for all $x \in S_d$, $f(x) \leq \lfloor\frac{t}{4}\rfloor - d$.

    Since, for all $d$ and $s \in S_d$ we have $f(s) \leq \lfloor\frac{t}{4}\rfloor - d$,
    if $S_{\lfloor\frac{t}{4}\rfloor} \neq \emptyset$, which we assumed, as otherwise the cop $C^i$ is interrupted,
    then there exists a vertex $s \in S_{\lfloor\frac{t}{4}\rfloor}$ such that $s \in N[q_{(i-1)\lfloor\frac{t}{4}\rfloor}]$
    while the cop $C^i$ occupies vertex $q_{(i-1)\lfloor\frac{t}{4}\rfloor}$ and it is the cop's turn.
    At this point the cop $C^i$ moves to such a vertex $s$ and as a result $C^i$ is in state (1)
    for the purpose of guarding $N[Q'_i]$.

    Finally, we claim that the robber never enters $N[Q'_i]$ during 
    the time it took the cop $C^i$ to employ this strategy and move to a position guarding $N[Q'_i]$.
    Recall that $C^{i-1}$ is guarding $N[Q_{i-1}]$
    at all prior steps of the argument, and recall that $q_{(i-1)\lfloor\frac{t}{4}\rfloor} \in V(Q_{i-1})$.
    Thus, the robber never enters $N[q_{(i-1)\lfloor\frac{t}{4}\rfloor}]$.
    Next, 
    we note that for all $s \in N(Q_i)$, 
    since the cop $C^i$ begins in state (2) for the purposes of guarding $N[Q_i]$
    \[
    \dist_H(s,R_0)\geq \lfloor \frac{t}{2}- \lfloor\frac{t}{4}\rfloor-1\rfloor.
    \]
    Hence, it would take the robber at least $\lfloor \frac{t}{2}- \lfloor\frac{t}{4}\rfloor-1\rfloor$ turns to reach any $s \in N(Q_i\cap Q'_i)$.
    Since 
    we have shown that if the cop has not begun guarding $N[Q'_i]$,
    then for all $s \in S_d$, $f(s) \leq \lfloor\frac{t}{4}\rfloor -d$,
    and our $s$ must have some vertex of $Q_i\cap Q'_i$, implying $f(s)\geq 1$,
    we note the robber only has $\lfloor \frac{t}{4}\rfloor -2$ turns to reach a vertex of $N(Q_i\cap Q'_i)$.
    Otherwise, the cop will be guarding $N[Q'_i]$,
    or the only vertices within distance $\lfloor \frac{t}{2}- \lfloor\frac{t}{4}\rfloor-1\rfloor$
    of the robber will be neighbours of $q_{(i-1)\lfloor\frac{t}{4}\rfloor}$.
    Since,
    \[
    \lfloor \frac{t}{4}\rfloor-2 <\lfloor \frac{t}{2}- \lfloor\frac{t}{4}\rfloor-1\rfloor
    \]
    the robber cannot reach a neighbour of $N(Q_i\cap Q'_i)$ before $N[Q'_i]$ is guarded.
    We conclude the desired claim that robber never enters $N[Q'_i]$ during 
    the time it took the cop $C^i$ to employ this strategy and move to a position guarding $N[Q'_i]$.
\hfill $\diamond$
\vspace{0.5cm}

This completes the proof that in Phase.2 the cops can redeploy from a path satisfying (i)-(iii)
while satisfying (I)-(IV).
As a result of the already provided arguments in Phase.2 this completes the proof that 
$\lceil\frac{k-2}{\lfloor\frac{t}{4}\rfloor}\rceil+1$
cops are sufficient to capture a robber in $G$.
\end{proof}

\section{Graphs with two Consecutive Forbidden Holes}

In this section we turn our attention from graphs with no short holes
to graphs that have two consecutive forbidden holes.
Such graph classes are necessarily much larger.
In exchange for this increased generality we achieve weaker bounds when $t$ is large.
It is unclear if the upper bounds we prove are tight for any $P_k$-free graphs, even those with no restriction on holes,
for $t\geq 2$.
Compared to the results in Section~2, these arguments are simpler.

\begin{lemma}\label{Lemma: Guarding with conscutive forbidden holes}
    Let $t\geq 2$ be an integer
    and $G$ a $(C_{t+2},C_{t+3})$-free
    graph.
    If $P = v_1\dots v_t$ is an induced path in $G$,
    for all $i \in \{1,\dots, t-1\}$ there is a cop $C^i$ on $v_i$,
    the robber is occupying a vertex in $V(G) \setminus N[P\setminus \{v_t\}]$,
    and it is the cop's turn,
    then $C^1,\dots, C^{t-1}$ are in a position to guard $N[P]$ in all subsequent turns.
\end{lemma}

\begin{proof}
    Let $t\geq 2$ and let $G$ be a $(C_{t+2},C_{t+3})$-free
    graph.
    Define $P$ as in the statement and assume that for all $i \in \{1,\dots, t-1\}$ there is a cop $C^i$ on $v_i$.
    Suppose the robber is occupying a vertex in $V(G) \setminus N[P\setminus \{v_t\}]$.

    For all vertices $x \in V(G)\setminus N[P]$
    let $f(x)=-1$ if $(N(x)\cap N(v_1))\setminus N(P\setminus \{v_1\}) \neq \emptyset$,
    let $f(x)=1$ if $(N(x)\cap N(v_t))\setminus N(P\setminus \{v_t\}) \neq \emptyset$,
    and let $f(x)=0$ otherwise.
    That is, $f(x)=-1$ if $x$ has a common neighbour with $v_1$
    that is private with with respect to $P$.
    With $f(x)=1$ defined analogously with $v_t$ replacing $v_1$.
    Meanwhile, $f(x)=0$ captures all remaining cases.
    
    We claim $f$ is well defined.
    Given the definition of $f$ it is sufficient to prove that 
    there exists no path
    $v_1uxwv_t$ where $x \notin N[P]$ and $u \notin N(P\setminus\{v_1\})$ and $w \notin N(P\setminus\{v_t\})$.
    Suppose not for contradiction.

    Then there exists a vertex $x \in V(G)\setminus N[P]$ adjacent to vertices
    $u \in N(v_1)$ and $w \in N(v_t)$
    such that $u \notin N[P\setminus \{v_1\}]$ and $w \notin  N[P\setminus \{v_t\}]$.
    Observe that $u$ and $w$ must be distinct.
    Also observe that $Pwxu$ is a cycle of length $t+3$
    whose only possible chordal is the edge $(u,w)$.
    If $(u,w)\in E(G)$, then $Pwu$ is a hole of length $t+2$, a contradiction.
    If $(u,w)\notin E(G)$, then $Pwxu$ is a hole of length $t+3$, also a contradiction.
    We conclude that $f$ is well defined.

    Now the cops $C^1,\dots, C^{t-1}$ guard $N[P]$
    as follows. Recall it is the cop's turn to move.
    Let $r\in V(G)\setminus N[P]$ be the current location of the robber.
    If $f(r)=0$ or $f(r)=-1$, then the cops do nothing.
    The robber cannot enter $N[P]$ on their next turn without being captured since, $f(r)=0$ or $f(r)=-1$
    implies that no vertex $N(r)\cap N(P)$ is a private neighbour of $v_t$.
    Since all other vertices of $P$ contain a cop, the robber cannot enter $N[P]$.
    Otherwise $f(r)=1$, implying no vertex $N(r)\cap N(P)$ is a private neighbour of  $v_1$.
    In this case, all cops $C^i$ slide one vertex to the right in $P$,
    leaving a cop $C^i$ on vertex $v_{i+1}$ for all $i$.
    The robber cannot move into $N(P)$ without moving next to a cop and being caught.

    After this move the cops are either in the same position or a symmetric position 
    to where they started and can repeat this strategy indefinitely in the obvious way.
    Thus, the cops can guard $N[P]$ from this starting positions as claimed.
    This completes the proof that if the robber begins outside $N[P]$, then 
    the cops can guard $N[P]$ for all subsequent turns.

    Suppose then that the robber occupies vertex $r \in N(v_t)\setminus N[P\setminus \{v_t\}]$.
    Then all cops slide to the right, as if $f(r)=1$.
    This forces the robber to move to a new vertex $R$.
    If $R\in N[P\setminus \{v_1\}]$, then there is a cop adjacent to $R$ and the robber is captured.
    If $R \in V(G) \setminus N[P]$, then from this position the cops can prevent the robber from ever 
    entering $N[P]$ again via the prior strategy.
    Otherwise, $R \in N(v_1)\setminus N[P\setminus \{v_1\}]$.
    Since the robber moved from $r$ to $R$ this implies $(r,R)\in E(G)$,
    hence $PrR$ is a hole of length $t+2$ a contradiction.
    This completes the proof.
\end{proof}

We are now prepared to prove Theorem~\ref{Thm: general n56}.
This proof proceeds almost identically to the proof of Theorem~\ref{Thm: No short hole},
so we shorten it to 
prevent redundancy.
As a result, we encourage the reader to understand the proof of 
Theorem~\ref{Thm: No short hole} prior to reading the following, abridged, proof.

\begin{proof}[Proof of Theorem~\ref{Thm: general n56}]
    Let $t\geq 3$ be an integer.
    and let $G$ be a fixed but arbitrary $(C_{t+2}, C_{t+3}, C_k, C_{k+1,\dots})$-free graph.
    We claim that 
    \[
    (t-1)\Big\lceil \frac{(k-2)}{t} \Big\rceil.
    \]
    cops have a winning strategy in $G$.

    The proof follows exactly the Phase.1 and Phase.2 strategy from the proof of Theorem~\ref{Thm: No short hole},
    but all intermediate steps are much easier,
    as we use $t-1$ cops to generate and guard each path $P_i$ and $Q_i$ rather than $1$ cop.
    Because of this, at all but the last step of each argument there is a cop on every 
    vertex of the respective induce path being constructed, so the robber can never move adjacent to the paths
    in question.
    Once the path is constructed Lemma~\ref{Lemma: Guarding with conscutive forbidden holes} implies the cops can guard 
    the path.
    The reader can easily verify that this strategy is sufficient.
    This completes the proof since $t-1$ cops are used to guard each induced path of order $t$.
\end{proof}

\section{Future Work}

To conclude the paper it feels valuable 
to add some conjectures.
Given the oddness of the current time these conjectures may be quickly solved.
However, they still seem to still be worth stating,
because they seem challenging, and their solutions may prove informative.
Also, they will hopefully serve the purpose of 
making others interested in this area.

\begin{conjecture}\label{Conj: no short no long}
    There exists an integer $N$ such that 
    for all $k\geq t\geq 4$,
    every 
    $(C_4,\dots, C_{t}, C_k, C_{k+1,\dots,})$-free
    graph $G$ has cop number
    \[
    c(G)  \leq   \frac{2k}{t} + N.
    \]
\end{conjecture}

It is well established, see for instance \cite{bonato2014note,joret2010cops},
that cop number is upper bounded by treewidth.
More strongly, one can easily observe that the
cop number is upper bounded by a function of the maximum domination number of a bag in a tree-decomposition,
minimized over all tree-decompositions.
To see this, suppose the min-maxed domination number here is $k$.
Use $k$ cops to dominate a bag, then move $k$ other cops to dominate the next bag closest to the robber.
By repeatedly doing this, with the two sets of $k$ cops switching roles, the cops will win.

Inspired by a question I was asked recently during research travel 
of the form `what is the nicest thing you could imagine being true of cop number and induced subgraphs?' it is worth proposing the following
`meta conjecture'.
Notice it might be true that 
hereditary class is too strong a notion here.
If that is the case, then it would be interesting
to understand if there is a natural, less general,
notion of graph class where the conjecture is true.

\begin{conjecture}\label{Conj: Meta-conj}
    Let $\mathcal{H}$ be a hereditary class of graphs.
    If
    for all $t$ the subset of $\mathcal{H}$ that is $(K_t,K_{t,t})$-free
    has bounded treewidth,
    then $\mathcal{H}$ has bounded cop number.
\end{conjecture}

The following question is stated to motivate
further interest in the problem, and because if correct it would be optimal up to additive constants.
Should such graphs exist their structure would be informative,
and if they do not exist, then the proof of such a result would likely be interesting.
Also note that if such graphs exist, they would serve as counterexamples to Conjecture~\ref{Conj: no short no long}.

\begin{question}\label{Conj: no C_4,C_5}
    Does there exists an integer $N$ such that for all $k$,
    there exists a $(C_4,C_5,C_k,C_{k+1},\dots)$-free graph with
    cop number at least $\frac{k}{2}-N$?
\end{question}

To conclude it seems worth formally stating a conjecture
that was gestured at in \cite{clow2027cops}.
As with many things in this paper, the conjecture is made about graphs with no long hole rather than $P_k$-free
graphs.
As has been mentioned a number of times, if true, then graphs in \cite{clow2026cops}
demonstrates this is tight up to the choice of $N$.

\begin{conjecture}\label{Conj: construction is correct}
    There exists an integer $N$ such that for all $k$,
    if $G$ is a $(C_k,C_{k+1},\dots)$-free graph, then 
    $c(G) \leq \frac{k}{2}+N$.
\end{conjecture}

\section*{AI Declaration}

I declare that no mathematical idea appearing in the paper was contributed by AI,
and that the 
text of this paper is written entirely by the author, 
with the exception of checking spelling and grammar, and double checking proofs for writing errors such as indexing mistakes.
This mathematical content of the paper, although not the writing, was completed prior to the middle of August 2026.
AI was used to survey the literature, along with standard search tools, when deciding
which special case of Conjecture~\ref{Conj: main} to work on.
Before making this paper public 
the author had
six agents of GPT-6 Astra, coordinated by a seventh,
attempt each of the conjectures and question proposed
for 1 hour.
This did not result in any of the conjectures or question being resolved.

\bibliographystyle{abbrv}
\bibliography{bib}

\begin{thebibliography}{10}

\bibitem{AIGNER1984}
M.~Aigner and M.~Fromme.
\newblock A game of \uppercase{C}ops and \uppercase{R}obbers.
\newblock {\em Discrete Applied Mathematics}, 8(1):1--12, 1984.

\bibitem{andreae1986pursuit}
T.~Andreae.
\newblock On a pursuit game played on graphs for which a minor is excluded.
\newblock {\em Journal of Combinatorial Theory, Series B}, 41(1):37--47, 1986.

\bibitem{bonato2013cops}
A.~Bonato and A.~Burgess.
\newblock Cops and robbers on graphs based on designs.
\newblock {\em Journal of Combinatorial Designs}, 21(9):404--418, 2013.

\bibitem{bonato2014note}
A.~Bonato, N.~E. Clarke, S.~Finbow, S.~Fitzpatrick, and M.-E. Messinger.
\newblock A note on bounds for the cop number using tree decompositions.
\newblock {\em Contributions to Discrete Mathematics}, 9(2), 2014.

\bibitem{bonato2011game}
A.~Bonato and R.~Nowakowski.
\newblock {\em The game of \uppercase{C}ops and \uppercase{R}obbers on graphs}.
\newblock American Mathematical Soc., 2011.

\bibitem{bose2026cops}
P.~Bose, L.~Esperet, J.~Hodor, G.~Joret, P.~Micek, and C.~Rambaud.
\newblock Cops and robber in graphs with bounded vertex cover number.
\newblock {\em arXiv preprint arXiv:2602.07435}, 2026.

\bibitem{bradshaw2023cop}
P.~Bradshaw, S.~A. Hosseini, B.~Mohar, and L.~Stacho.
\newblock On the cop number of graphs of high girth.
\newblock {\em Journal of Graph Theory}, 102(1):15--34, 2023.

\bibitem{chudnovsky2024cops}
M.~Chudnovsky, S.~Norin, P.~D. Seymour, and J.~Turcotte.
\newblock \uppercase{C}ops and \uppercase{R}obbers on $\uppercase{P}_5$-free graphs.
\newblock {\em SIAM Journal on Discrete Mathematics}, 38(1):845--856, 2024.

\bibitem{clow2025cops}
A.~Clow, M.~A. Huggan, and M.-E. Messinger.
\newblock Cops and attacking robbers with cycle constraints.
\newblock {\em Discrete Applied Mathematics}, 373:327--342, 2025.

\bibitem{clow2027cops}
A.~Clow and E.~Meger.
\newblock Cops and robbers on graphs with path constraints.
\newblock {\em Discrete Mathematics}, 350(1):115392, 2027.

\bibitem{clow2026cops}
A.~Clow and I.~Zaguia.
\newblock Cops and robbers, clique covers, and induced cycles.
\newblock {\em Discrete Mathematics}, 349(11):115252, 2026.

\bibitem{de2024guarding}
S.~G.~H. de~la Maza and B.~Mohar.
\newblock Guarding isometric subgraphs and cops and robber in planar graphs.
\newblock {\em European Journal of Combinatorics}, 119:103809, 2024.

\bibitem{frankl1987cops}
P.~Frankl.
\newblock Cops and robbers in graphs with large girth and cayley graphs.
\newblock {\em Discrete Applied Mathematics}, 17(3):301--305, 1987.

\bibitem{gupta2023cops}
U.~K. Gupta, S.~Mishra, and D.~Pradhan.
\newblock Cops and robber on subclasses of p5-free graphs.
\newblock {\em Discrete Mathematics}, 346(6):113353, 2023.

\bibitem{PkNoClawFork2026}
U.~K. Gupta, D.~Pradhan, and R.~R. Swain.
\newblock Cop number of $p_k$-free graphs without a claw or a fork.
\newblock {\em arXiv preprint arXiv:2610.20129}, 2026.

\bibitem{javadi2021game}
R.~Javadi and A.~Momeni.
\newblock The game of cops and robber on (claw, even-hole)-free graphs.
\newblock {\em arXiv preprint arXiv:2112.07503}, 2021.

\bibitem{joret2010cops}
G.~Joret, M.~Kami{\'n}ski, and D.~O. Theis.
\newblock The \uppercase{C}ops and \uppercase{R}obber game on graphs with forbidden (induced) subgraphs.
\newblock {\em Contributions to Discrete Mathematics}, 5(2), 2010.

\bibitem{kenter2025improved}
F.~Kenter, E.~Meger, and J.~Turcotte.
\newblock Improved bounds on the cop number when forbidding a minor.
\newblock {\em Journal of Graph Theory}, 108(3):620--646, 2025.

\bibitem{kosowski2015k}
A.~Kosowski, B.~Li, N.~Nisse, and K.~Suchan.
\newblock k-chordal graphs: From cops and robber to compact routing via treewidth.
\newblock {\em Algorithmica}, 72(3):758--777, 2015.

\bibitem{lehner20262}
F.~Lehner and A.~Sun.
\newblock The 2-cops-move cop number of graphs on the torus, klein bottle, and projective plane.
\newblock {\em arXiv preprint arXiv:2607.21821}, 2026.

\bibitem{liu2019cop}
M.~Liu.
\newblock The cop number of graphs with forbidden induced subgraphs.
\newblock {\em arXiv preprint arXiv:1908.11478}, 2019.

\bibitem{luczak2010chasing}
T.~{\L}uczak and P.~Pra{\l}at.
\newblock Chasing robbers on random graphs: zigzag theorem.
\newblock {\em Random Structures \& Algorithms}, 37(4):516--524, 2010.

\bibitem{nowakowski1983vertex}
R.~Nowakowski and P.~Winkler.
\newblock Vertex-to-vertex pursuit in a graph.
\newblock {\em Discrete Mathematics}, 43(2-3):235--239, 1983.

\bibitem{pralat2010does}
P.~Pralat.
\newblock When does a random graph have constant cop number?
\newblock {\em Australasian Journal of Combinatorics}, 46:285--296, 2010.

\bibitem{pralat2016meyniel}
P.~Pra{\l}at and N.~Wormald.
\newblock Meyniel's conjecture holds for random graphs.
\newblock {\em Random Structures \& Algorithms}, 48(2):396--421, 2016.

\bibitem{quilliot1983problemes}
A.~Quilliot.
\newblock Problemes de jeux, de point fixe, de connectivit{\'e} et de repr{\'e}sentation sur des graphes, des ensembles ordonn{\'e}s et des hypergraphes.
\newblock {\em These d’Etat, Universit{\'e} de Paris VI}, pages 131--145, 1983.

\bibitem{sivaraman2019application}
V.~Sivaraman.
\newblock An application of the \uppercase{G}y{\'a}rf{\'a}s path argument.
\newblock {\em Discrete Mathematics}, 342(8):2306--2307, 2019.

\bibitem{sivaraman2019cop}
V.~Sivaraman.
\newblock Cop number of graphs without long holes.
\newblock {\em arXiv preprint arXiv:2001.00477}, 2019.

\bibitem{sivaraman2019copK2}
V.~Sivaraman and S.~Testa.
\newblock Cop number of $2\uppercase{K}_2$-free graphs.
\newblock {\em arXiv preprint arXiv:1903.11484}, 2019.

\bibitem{turcotte2022cops}
J.~Turcotte.
\newblock \uppercase{C}ops and \uppercase{R}obbers on 2$\uppercase{K}_2$-free graphs.
\newblock {\em Discrete Mathematics}, 345(1):112660, 2022.

\bibitem{west2001introduction}
D.~B. West.
\newblock {\em Introduction to Graph Theory}, volume~2.
\newblock Prentice hall Upper Saddle River, 2001.

\end{thebibliography}

\end{document}